\documentclass[12pt]{amsart}
\usepackage[T1]{fontenc}
\usepackage{microtype}
\usepackage{amsmath}
\usepackage{xcolor}
\usepackage{amsmath}

\usepackage[scaled=.98,sups,osf]{XCharter}
\usepackage[charter,vvarbb,scaled=1.07]{newtxmath}
\usepackage[type1]{cabin}
\usepackage[scaled=1.04]{inconsolata}

\usepackage[bb=ams,scr=euler,cal=boondoxo]{mathalpha}
\usepackage[backref=page]{hyperref}
\usepackage{tikz,float,caption}
\usetikzlibrary{arrows.meta,calc,decorations.markings,patterns,cd,patterns.meta}
\usepackage[margin=1.25 in,top=1.2in, bottom=1 in]{geometry}

\usepackage{setspace}
\usepackage{enumitem}
\setenumerate{
  wide=10pt,
  leftmargin=0pt,
  labelwidth=*,
  label=(\roman*),
  topsep=0pt,
  listparindent=0pt,
  parsep=4pt}

\usepackage{thmtools}
\declaretheoremstyle[headfont=\normalsize\normalfont\bfseries,notefont=\mdseries, notebraces={(}{)},bodyfont=\normalfont,postheadspace=0.5em, spaceabove=3mm, spacebelow=2mm]{basicstyle}
\declaretheoremstyle[headfont=\normalsize\normalfont\bfseries,notefont=\mdseries,
notebraces={(}{)},bodyfont=\normalfont\itshape,postheadspace=0.5em, spaceabove=3mm, spacebelow=2mm]{italstyle}

\declaretheorem[style=basicstyle,name=Conjecture]{conjecture}

\declaretheorem[name=Definition,style=italstyle,style=basicstyle]{defn}

\declaretheorem[style=italstyle,name=Theorem,numberwithin=section]{theorem}
\declaretheorem[style=italstyle,name=Corollary,sibling=theorem]{cor}
\declaretheorem[style=italstyle,name=Claim,sibling=theorem]{claim}
\declaretheorem[style=italstyle,name=Proposition,sibling=theorem]{prop}
\declaretheorem[style=italstyle,name=Lemma,sibling=theorem]{lemma}
\declaretheorem[style=basicstyle,name=Proof,numbered=no,qed=$\square$]{preproof}
\renewenvironment{proof}{\preproof}{\endpreproof}

\makeatletter
\newcommand{\abs}[1]{\left|#1\right|}
\newcommand{\bd}{\partial}
\renewcommand{\d}{\mathrm{d}}

\newcommand{\id}{\mathrm{id}}
\newcommand{\im}{\mathrm{im}}
\newcommand{\intprod}{\mathbin{{\tikz{\draw[line width=0.6](-0.1,0)--(0.1,0)--(0.1,0.2)}\hspace{0.5mm}}}}
\newcommand{\norm}[1]{\left\lVert#1\right\rVert}
\newcommand{\pd}[2]{\frac{\partial #1}{\partial #2}}
\newcommand{\pr}{\mathrm{pr}}
\newcommand{\R}{\mathbb{R}}

\newcommand{\F}{\mathbb{F}}
\def\@secnumfont{\bfseries}
\renewcommand\section{\@startsection{section}{1}{0pt}{-3.5ex \@plus -1ex \@minus -.2ex}{2.3ex \@plus.2ex}{\centering\itshape}}
\newcommand{\set}[1]{\left\{#1\right\}}
\renewcommand{\subsection}{\@startsection{subsection}{2}%
  \z@{.5\linespacing\@plus.7\linespacing}{-.5em}%
  {\normalfont\itshape}}

\newcommand{\w}{\omega}
\newcommand{\Z}{\mathbb{Z}}

\newcommand{\pss}{\mathrm{PSS}}
\newcommand{\Ham}{\mathrm{Ham}}

\newcommand{\Spec}{\mathrm{Spec}}
\newcommand{\cf}{\mathrm{CF}}
\newcommand{\hf}{\mathrm{HF}}

\newcommand{\fix}{\mathrm{Fix}}

\newcommand{\Hof}{\mathrm{Hof}}
\newcommand{\cl}{\mathrm{cl}}
\newcommand{\K}{\mathbb{K}}
\newcommand{\om}{\omega}
\newcommand{\cD}{\mathcal{D}}
\newcommand{\cA}{\mathcal{A}}

\newcommand{\cJ}{\mathcal{J}}

\makeatother

\title[Lagrangian intersections and the spectral norm]{Lagrangian intersections and the spectral norm in convex-at-infinity symplectic manifolds}
\author{Habib Alizadeh}
\author{Marcelo S. Atallah}
\author{Dylan Cant}

\begin{document}

\begin{abstract}
  Given a compact Lagrangian $L$ in a semipositive convex-at-infinity symplectic manifold $W$, we establish a cup-length estimate for the action values of $L$ associated to a Hamiltonian isotopy whose spectral norm is smaller than some $\hbar(L)$. When $L$ is rational, this implies a cup-length estimate on the number of intersection points. This Chekanov-type result generalizes a theorem of Kislev and Shelukhin proving non-displaceability in the case when $W$ is closed and monotone. The method of proof is to deform the pair-of-pants product on Hamiltonian Floer cohomology using the Lagrangian $L$.
\end{abstract}

\maketitle

\section{Introduction and main results}

\subsection{Introduction}
\label{sec:introduction}
Let $(W,\w)$ be a symplectic manifold and $L\subset W$ a closed Lagrangian submanifold. Understanding when a Hamiltonian diffeomorphism $\phi$ can displace $L$, and quantifying the intersections when it cannot, has been one of the driving forces of symplectic topology ever since Arnol'd's famous conjectures were formulated; see \cite{arnold65,arnold13book}. Let us denote by $\mathrm{Ham}_{c}(W,\omega)$ the group of compactly supported Hamiltonian diffeomorphisms, i.e., those diffeomorphisms $\phi$ which appear as the time-one map $\phi=\phi_{1}$ of a compactly supported Hamiltonian isotopy $\phi_{t}$; recall that this means the non-autonomous vector field $X_{t}$ generating $\phi_{t}$ is $\omega$-dual to an exact one-form.

The Lagrangian version of one of the conjectures in the particular case of cotangent bundles states the following:
\begin{conjecture}[Arnol'd]\label{conjecture:arnold}
  For every compactly supported Hamiltonian diffeomorphism $\phi$ of $T^{*}L$, the number of intersection points $\phi(L)\cap L$ is bounded from below by the minimal number of critical points of a smooth function on $L$.
\end{conjecture}

When the intersection is transverse, the conjectured lower-bound is replaced by the \emph{Morse number} of $L$, i.e., the minimal number of critical points of a Morse function on $L$. In this direction, Gromov proves in his groundbreaking work \cite{gromov85} the existence of at least one intersection point of any closed exact Lagrangian $L^\prime\subset T^*L$ with the zero-section, and then sets $L'=\phi(L)$ to conclude $\phi(L)\cap L\neq\emptyset$.

As stated, Conjecture \ref{conjecture:arnold} remains open; however, in a classic result \cite{hofer85}, Hofer proves a slightly weaker version of the conjecture where the lower bound is replaced with one plus the cup-length $\cl_{\K}(L)$ of $L$ with coefficients on a base field $\K$, where:
\begin{equation*}
  \cl_{\K}(L)=\max\{k\,|\,\exists\,a_1,\dots,a_k\in H^{>0}(L;\K)\,\,\text{such that}\,\, a_1\cup\cdots\cup a_k\neq0\}.
\end{equation*}
In another celebrated result \cite{laudenbach-sikorav85}, Laudenbach and Sikorav showed that, in the transverse case, the number of intersection points is bounded from below by the total Betti number of $L$.

A Lagrangian submanifold $L\subset W$ is called \textit{weakly-exact} if $\w(\pi_2(W,L))=0$. More generally, $L$ is called \textit{rational} If $\w(\pi_2(W,L))\subset\R$ is a discrete subgroup, in which case we denote the positive generator by $\rho_L$. When $(W,\w)$ is a tame symplectic manifold (see \S\ref{sec:tame-sympl-manif}), Gromov shows in \cite[{2.3.$\mathrm{B}_{3}'$}]{gromov85} that a weakly-exact Lagrangian submanifold $L$ is non-displaceable. Moreover, if $W$ is closed, the foundational works \cite{floer_morse, floer_cuplength} of Floer imply that the number of intersection points $\phi(L)\cap L$ is at least $\cl_{\K}(L)+1$ in general, and $\dim H_{*}(L)$ when the intersection is transverse; see also \cite{hofer1988lusternik}. Recently, this cup-length estimate has been established by \cite{hirschi-porcelli} for generalized cohomology theories.

Nonetheless, the existence of small displaceable Lagrangian tori in every symplectic manifold indicates that a generalization of Conjecture \ref{conjecture:arnold} beyond the weakly-exact setting requires additional hypothesis; one can, e.g., require that the Hamiltonian diffeomorphism is close to the identity in some sense.

In this direction, Polterovich \cite{polterovich_lag_displacement_energy} proved that if $L$ is rational and $W$ is tame, then the \emph{Hofer norm} of any $\phi$ displacing $L$ is at least $\rho_L/2$; i.e., $\phi(L)\cap L\neq\emptyset$ provided $\norm{\phi}_{\Hof}<\rho_{L}/2$; see \cite{hofer_energy,lalonde_mcduff95,polterovich_book} for discussion of the Hofer norm.
This result was sharpened by Chekanov \cite{chekanov_1998}, as follows. For an $\omega$-tame almost complex structure $J$ on $W$, define $\hbar(J,L)>0$ to be the minimal symplectic area of a non-constant $J$-holomorphic disk with boundary on $L$ or sphere in $W$, and set:
\begin{equation}\label{eq:hbar-constant}
  \hbar(L)=\sup_{J\in\cJ}\hbar(J,L),
\end{equation}
where $\cJ$ denotes the space of $\w$-tame almost complex structures on $W$; see \S\ref{sec:tame-sympl-manif}. Chekanov showed that if $\norm{\phi}_{\Hof}<\hbar(L)$, then $\phi(L)\cap L\neq\emptyset$, and the number of intersection points is bounded from below by $\dim_{\F_2}H_{*}(L;\F_2)$ provided the intersection is transverse; see also \cite{liu05}.

Spectral invariants provide a way of defining a \emph{spectral norm} on $\Ham_c(W,\w)$ which is bounded from above by the Hofer norm. They were introduced in symplectic topology by Viterbo \cite{viterbo92} via generating functions and, from a Floer theoretic perspective, by Schwarz \cite{schwarz_spectral_invariants} and Oh \cite{oh-2005-birkhauser, oh-2005-duke} (in the closed setting) and by Frauenfelder and Schlenk \cite{frauenfelder_schlenk} for convex-at-infinity symplectic manifolds (as defined in \S\ref{sec:convex_end}); see also \cite[\S5.4]{hofer-zehnder-94} and \cite[\S1.5.B]{bialy-polterovich-duke-1994}. In short, for a Hamiltonian system $\phi_{t}$, one associates real-valued measurements $c(\alpha,\phi_{t})$ to classes $\alpha$ in the (quantum) cohomology of $W$; the definition is as a ``min-max'' action value of the Floer cohomology class representing the image of $\alpha$ under the map in \cite{pss}; we review their construction in \S\ref{sec:HFC-intro}.
The \emph{spectral norm} of a compactly supported Hamiltonian diffeomorphism $\phi$ is defined by:
\begin{equation}\label{eq:defin-spectral-norm}
  \gamma(\phi)=\inf_{\phi_{1}=\phi}-c(1,\phi_{t})-c(1,\phi_{t}^{-1}),
\end{equation}
and it satisfies $\gamma(\phi)\leq\norm{\phi}_{\Hof}$.

When $(W,\w)$ is a closed weakly-monotone symplectic manifold, Kislev and Shelukhin showed in \cite[Theorem E]{kislev_shelukhin} that if $\gamma(\phi)<\hbar(L)$ then $\phi$ does not displace $L$ and, if $\phi(L)\cap L$ is transverse, then $\#(\phi(L)\cap L)\geq\dim_{\F_2}H_{*}(L;\F_2)$, sharpening Chekanov's result in this setting. The reason this improvement is possible boils down to the observation that the Floer continuation maps: \[\mathfrak{c}:\cf(\phi_t)\rightarrow \cf(\psi_t)\] are chain-homotopic to the multiplication operators: \[\mu_2(x,-):\cf(\phi_t)\rightarrow \cf(\psi_t)\] given by taking the product with a cocycle $x\in \cf(\psi_t\circ\phi^{-1}_t)$ representing the image of the unit under the PSS map.

In \cite[Remark 50]{kislev_shelukhin} it is suggested that the cup-length estimate for a suitable choice of coefficient field $\K$ should hold whenever $\gamma(\phi)<\hbar(L)$; see also \cite{gong_cup}. Proving such a cup-length estimate in the convex-at-infinity setting is the main goal of this paper.

\subsection{Main results}
\label{sec:main_results}

Let $(W,\w)$ be a semipositive convex-at-infinity symplectic manifold and $L$ a closed connected Lagrangian. The class of convex-at-infinity symplectic manifolds contains all Liouville manifolds and compact symplectic manifolds; see \S\ref{sec:convex_end}.

Recall that, to a compactly supported Hamiltonian system $\phi_t$, one can associate an action functional $\cA_{\phi_t}$ on the covering space of ``capped'' paths $x(t)$ with endpoints of $L$; see \S\ref{sec:hamiltonian_action} for the definitions. The critical points of $\cA_{\phi_t}$ are in bijective correspondence with the capped Hamiltonian chords.
\begin{theorem}\label{thm:main}
Suppose $\phi$ is a compactly supported Hamiltonian diffeomorphism such that $\gamma(\phi)<\hbar(L)$. Then,\[\phi(L)\cap L\neq\emptyset.\] Moreover, if the intersection points are isolated, then for all Hamiltonian system $\phi_t$ with $\phi_1=\phi$ there exists an interval of length $\gamma(\phi)$ containing at least $\cl_{\F_2}(L)+1$ critical values of $\cA_{\phi_{t}}$.
\end{theorem}

When $L$ is a rational Lagrangian submanifold with rationality constant $\rho_{L}$ the action value of a path is well-defined modulo $\rho_L$. Since $\rho_L\leq\hbar(L)$, Theorem \ref{thm:main} yields:
\begin{cor}\label{cor:rational}
Let $L$ be a compact rational Lagrangian submanifold. The cup-length estimate $\#(\phi(L)\cap L)\geq \cl_{\F_2}(L)+1$ holds for all compactly supported Hamiltonian diffeomorphisms $\phi$ satisfying $\gamma(\phi)<\rho_L$.
\end{cor}
When $(W,\w)$ is a closed rational semipositive symplectic manifold our result sharpens that of \cite[Theorem 1.1]{schwarz98} by replacing the Hofer norm with the spectral norm. More precisely, we obtain the following:
\begin{cor}
Let $(W,\w)$ be a compact semipositive symplectic manifold and suppose that $\w(\pi_2(W))=\rho_W\Z$. If $\phi$ is a Hamiltonian diffeomorphism satisfying $\gamma(\phi)\leq\rho_W$ then $\#\fix(\phi)\geq\cl_{\F_2}(W)+1$.
\end{cor}
\begin{proof}
Consider $(W\times W, \w\oplus-\w)$ with the diagonal Lagrangian $\Delta$. We first show that $\Delta$ is a rational with rationality constant $\rho_\Delta=\rho_W$. Let $A$ be a relative class in $\pi_2(W\times W, \Delta)$ represented by:
\begin{equation*}
  V\colon (D,\bd D)\rightarrow (W\times W,\Delta)\hspace{.3cm}\text{given by}\hspace{.3cm}z\mapsto (v_1(z),v_2(z)),
\end{equation*}
where $v_1$ and $v_2$ are the projections of $V$ onto the first and second factors. Note that if $z\in\bd D$ then $V(z)\in\Delta$; in particular, we have $v_1(z)=v_2(z)$. Consider the piecewise smooth sphere $u=v_{1}\#(-v_{2})$ obtained by gluing $v_1$ and $v_2$ along their common boundary (and reversing the orientation of $v_2$). The symplectic area of $u$ in $W$ equals the symplectic area of $V$ in $W\times W$, hence $\rho_{\Delta}\Z\subset\rho_{W}\Z$. For the reverse inclusion, observe that every smooth sphere decomposes as $v_1\#(-v_2)$ where $v_1|_{\bd D}=v_2|_{\bd D}$, and the previous argument can be run in reverse to conclude $\rho_{W}\Z=\rho_{\Delta}\Z$, as desired.

Next, set $\Phi=\id\times\phi$ and consider the Lagrangian submanifold $\Phi(\Delta)$ of $W\times W$. The intersection points $\Phi(\Delta)\cap\Delta$ correspond bijectively to the fixed points of $\phi$. We appeal to the product formula for spectral invariants in \cite[Theorem 5.1]{entov-polterovich} to conclude that:
\begin{align*}
  \gamma(\Phi)
  &=\inf_{\Phi_{1}=\Phi}-c(1,\Phi_{t})-c(1,\Phi_{t}^{-1})\\
  &\leq\inf_{\phi_{1}=\phi}-c(1,\phi_{t}\times\id)-c(1,\phi_{t}^{-1}\times\id)\\
  &=\inf_{\phi_{1}=\phi}-c(1,\phi_{t})-c(1,\phi_{t}^{-1})=\gamma(\phi).
\end{align*}
Hence, $\gamma(\Phi)\leq\gamma(\phi)\leq\rho_{W}=\rho_{\Delta}$, and we can therefore apply Corollary \ref{cor:rational} to obtain:
\begin{align*}
 \#\fix(\phi)=\#(\Phi(\Delta)\cap\Delta)\geq\cl_{\F_2}(\Delta)+1=\cl_{\F_2}(W)+1,
\end{align*}
which concludes the proof of the corollary.
\end{proof}

\subsection{Proof overview}
\label{sec:overview}
Before delving into an overview of the proof of Theorem \ref{thm:main}, let us first examine a simpler case to illustrate the underlying principles in our approach, while pointing to the difficulties that arise in the more general setting. Suppose that $L\subset W$ is a closed weakly-exact Lagrangian submanifold with cup-length $\cl_{\F_2}(L)=k$. To prove the cup-length estimate it is enough to show that the action spectrum: \[\Spec(\phi_t,L)=\mathrm{Crit}(\cA_{\phi_t})\] has at least $k+1$ values (for any Hamiltonian system $\phi_t$ generating $\phi$). Indeed, the weakly-exact condition implies that the action value of a chord is independent of the choice of capping. In contrast, for rational $L$, the action $\cA_{\phi_t}$ is defined modulo $\rho_L$. In the rational case one can still ensure the existence of at least $k+1$ Hamiltonian chords by showing that action values belong to an interval of length at most $\rho_L$; this excludes contributions of different cappings of the same chord.

To obtain a strictly decreasing sequence of $k+1$ action values $a_{0}>\dots>a_{k}$, it is sufficient to have a chain of $k$ non-stationary Floer strips $u_{1},\dots,u_{k}$ with boundary on $L$; here \emph{non-stationary} means the energy of the strip is non-zero, and \emph{chain} means the positive asymptotics of $u_{j}$ equals the negative asymptotic of $u_{j+1}$. See Figure \ref{fig:sequence-floer-strip} below for an illustration of such a chain.

Given such a chain, the sum of the energies of the $u_{j}$ bounds the difference $a_{0}-a_{k}$. Thus, if it is possible to construct chains of non-constant Floer strips satisfying a total energy bound less than $\rho_{L}$, one obtains the desired chain of action values.

One way to conclude a non-constant Floer strip $u$ for the system $\phi_{t}$ relative the Lagrangian $L$ is to require that $u(0,0)$ lies on a smooth cycle $f:P\to L$ which is disjoint from $\phi_{1}(L)$. Such curves are used to define a \emph{cap-action} of $f$ on the Lagrangian Floer cohomology. In the weakly exact setting, well-known arguments using this cap-action explain how to construct chains of Floer strips whose length is the cup-length; we recall the arguments in \S\ref{sec:lagr-cap-acti}.

The main difficulty in generalizing this argument is the bubbling of $J$-holomorphic disks. For one, the bubbling phenomenon impedes us from defining Lagrangian Floer cohomology and Lagrangian cap-action.

In \S\ref{sec:algebra} and \S\ref{sec:moduli-approach} we explain how to construct chains of Floer strips of length $k$, with total action bound $a_{0}-a_{k}\le \gamma(\phi_{t})$. The approach in \S\ref{sec:algebra} is based on the module action of Hamiltonian Floer cohomology on Lagrangian Floer cohomology considered in \cite{kislev_shelukhin}. One only considers Lagrangian Floer cohomology defined in action windows smaller than the disk bubbling threshold. In \S\ref{sec:moduli-approach}, we explain how to deform the pair-of-pants product on Hamiltonian Floer cohomology using a compact Lagrangian in such a way which circumvents the need to consider Lagrangian Floer cohomology entirely, while still concluding a configuration of strips as in Figure \ref{fig:sequence-floer-strip}.

\subsubsection{The Lagrangian cap-action in the weakly exact case}
\label{sec:lagr-cap-acti}

In the weakly-exact setting, Lagrangian Floer cohomology $\hf(L,\phi_t)$ is well-defined, since there is no disk bubbling; see, e.g., \cite{kislev_shelukhin}. The PSS isomorphism provides an identification: \[\pss_{L,\phi_t}:H(L)\rightarrow \hf(L,\phi_t).\] Moreover, every bordism class $\Pi$ of maps $f:P\rightarrow L$ has a corresponding Lagrangian cap action map: \[\mathrm{cap}_{\Pi}:\hf(L,\phi_t)\rightarrow\hf(L,\phi_t).\] On the chain level, the map is defined by picking a representative $f$ of $\Pi$ and counting Floer strips $u:\R\times[0,1]\rightarrow W$ satisfying $u(\R\times\{0\}),u(\R\times\{1\})\in L$ and $u(0,0)\in f(P)$. The cap action is associative: \[\mathrm{cap}_{\Pi\cap\Pi^\prime}=\mathrm{cap}_{\Pi}\circ\mathrm{cap}_{\Pi^\prime},\] and is compatible with the PSS isomorphism: \[\mathrm{cap}_{\Pi}(\pss_{L,\phi_t}(\Pi^\prime))=\pss_{L,\phi_t}(\Pi\cap\Pi^\prime);\] see e.g., \cite[\S4]{le-ono} and \cite{floer_cup} for associativity and, e.g., \cite[\S3]{pss} and \cite[\S2.3]{schwarz_spectral_invariants} for compatibility with PSS. The weakly-exact open-string case is handled analogously to the closed-string case.

It is a convenient fact that the cup-length in unoriented bordism is the same as cup-length in singular cohomology with $\F_2$ coefficients; see \cite[Theorem III.2]{thom-quelques-proprietes}, and, e.g., \cite[\S3.4]{wilkins_quantum_steenrod} and \cite[Theorem B]{buoncristiano-hacon}. Therefore, there exist bordism classes $\Pi_1,\dots,\Pi_k$ of maps $f_{i}:P_i\rightarrow L$ of positive codimension, for $i=1,\dots,k$, such that the intersection product $\Pi_1\cap\cdots\cap\Pi_k$ equals the point class $[\mathrm{pt}]$. Becausethe codimension is $\ge 1$, we can make the images of $f_{i}$ disjoint from $\phi_{1}(L)\cap L$ (assuming the intersections form an isolated set).

On the one hand, the cap action of the point class is non-trivial since: \[\mathrm{cap}_{[\mathrm{pt}]}(\pss_{L,\phi_t}([L]))=\pss_{L,\phi_t}([\mathrm{pt}]\cap[L])=\pss_{L,\phi_t}([\mathrm{pt}]).\] On the other hand, by associativity, we have:\[\mathrm{cap}_{[\mathrm{pt}]}=\mathrm{cap}_{\Pi_{1}}\circ\cdots\circ\mathrm{cap}_{\Pi_{k}}.\] The non-triviality of the above chain of compositions implies, in particular, that there exists a sequence of $k$ Floer strips with point constraints as illustrated in Figure \ref{fig:sequence-floer-strip}.

 \begin{figure}[H]
  \centering
  \begin{tikzpicture}
    \draw (0,0) rectangle (3,1);
    \draw (3.5,0) rectangle +(3,1);
    \draw (7.3,0) rectangle +(3,1);
    \path (3,0.5)node[left]{$\gamma_{1,+}$}--+(3.5,0)node[left]{$\gamma_{2,+}$}--+(7.3,0)node[left]{$\gamma_{k,+}$};
    \path (0,0.5)node[right]{$\gamma_{1,-}$}--+(3.5,0)node[right]{$\gamma_{2,-}$}--+(7.3,0)node[right]{$\gamma_{k,-}$};
    \path (6.5,0.5)node[right]{$\cdots$};
    \path[every node/.style={draw,circle,fill=black,inner sep=1pt}] (1.5,0)node{}--+(3.5,0)node{}--+(7.3,0)node{};
    \path (1.5,0)node[below]{$f_{1}(P_{1})$}--+(3.5,0)node[below]{$f_{2}(P_{2})$}--+(7.3,0)node[below]{$f_{k}(P_{k})$};
  \end{tikzpicture}
  \caption{A sequence of Floer strips with point constraints and uniformly bounded energy converges-up-to-breaking to a sequence of Floer strips, satisfying action bounds $\cA_{\phi_t}(\gamma_{j+1,-})\le \cA_{\phi_t}(\gamma_{j,+})<\cA_{\phi_t}(\gamma_{j,-})$.}
  \label{fig:sequence-floer-strip}
\end{figure}
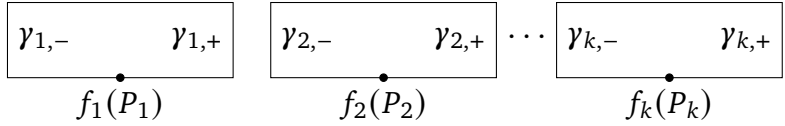
The strict inequalities $\cA_{\phi_t}(\gamma_{j,-})<\cA_{\phi_t}(\gamma_{j,+})$, for all $j\in\{1,\dots,k\}$, follow from the fact that each Floer strip is non-stationary because of the incidence constraint. Thus, there are at least $k+1$ action values, which concludes the argument.

\subsubsection{An algebraic approach}
\label{sec:algebra}

The approach in this section is heavily inspired by \cite[Theorem E]{kislev_shelukhin}. It relies on defining Lagrangian Floer cohomology, along with zero curvature operations, in an action window that is small enough to prevent bubbling yet sufficiently large to detect cohomological information of the Lagrangian. While this approach can likely be generalized to the convex-at-infinity setting, in this section we restrict ourselves to the case where $L$ is a monotone Lagrangian of a closed symplectic manifold $(W,\om)$ since the tools required have been carefully defined in \cite{kislev_shelukhin}.

Given a Hamiltonian system $\phi_t$ and an $\om$-compatible almost complex structure $J$, for each interval $I$ of length $\abs{I}<\hbar(J,L)$, let $\cf(\phi_t,L;\cD)^{I}$ be the Floer complex generated by capped (contractible) chords whose action values belong to the interval $I$; one should suppose that the endpoints of $I$ are not action values of chords. For generic perturbation data $\cD$, the differential is well defined since bubbling is prevented by the narrow action window. We denote by $\hf(\phi_t,L;\cD)^{I}$ the corresponding cohomology. As in \S\ref{sec:lagr-cap-acti}, it is possible to define a Lagrangian cap-action associated to a bordism class $\Pi$. Depending on the action window, these maps could very well be trivial.

A new input compared to \S\ref{sec:lagr-cap-acti} is the following: to a cocycle $z\in\cf(\psi_t; \cD)$ in the Hamiltonian Floer cohomlogy of action $\cA_{\psi_t}(z)=c$, there corresponds a multiplication operation:
\begin{equation*}
  [\mu(-,z)]:\hf(L,\phi_t;\cD)^{I}\rightarrow\hf(L,\psi_t\circ\phi_t;\cD)^{I+c+\epsilon},
\end{equation*}
where $\epsilon>0$ is an error term related to the perturbation data $\mathscr{D}$. In short, $\mu(-,z)$ is defined by counting rigid one-punctured Floer strips with Lagrangian boundary conditions, whose ends are asymptotic to Hamiltonian chords of $\phi_t$ and $\psi_t\circ\phi_t$ and whose interior puncture is asymptotic to a one-periodic orbit of $\psi_t$ belonging to the linear combination expressing $z$; see Figure \ref{fig:one-tower-module-action}.

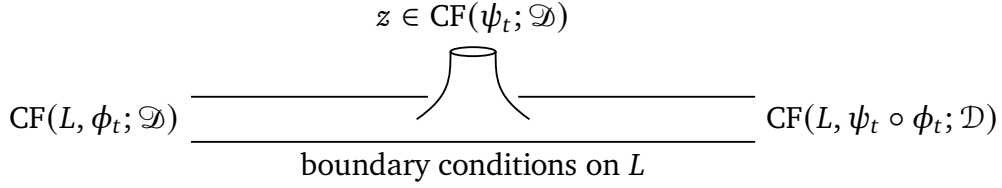
\begin{figure}[H]
  \centering
  \begin{tikzpicture}[yscale=0.6,xscale=1.5]
    \draw[line width=.7] (0,0)--(5,0) (0,1)--(2.1,1) (2.9,1)--(5,1);
    \draw[line width=.7] (2,0.5)to[out=60,in=-90](2.3,2) (3,0.5)to[out=120,in=-90](2.7,2) (2.5,2) circle (0.2 and 0.1)node(X){};
    \node[shift={(0,0.1)},above] at (X) {$z\in \mathrm{CF}(\psi_{t};\mathcal{D})$};
    \node at (0,0.5) [left]{$\mathrm{CF}(L,\phi_{t};\mathcal{D})$};
    \node at (5,0.5) [right]{$\mathrm{CF}(L,\psi_{t}\circ \phi_{t};\mathscr{D})$};
    \node at (2.5,0)[below]{boundary conditions on $L$};
  \end{tikzpicture}
  \caption{The module action of Hamiltonian Floer cohomology on Lagrangian Floer cohomology; see \cite[\S4.1]{kislev_shelukhin}.}
  \label{fig:one-tower-module-action}
\end{figure}

The goal is to find an action interval $I$ of length less than $\hbar(L)$ on which composing $k$ times the restricted Lagrangian cap-action is a non-trivial operation, as this guarantees the existence of a chain of $k$ Floer strips exactly as in \S\ref{sec:lagr-cap-acti}.

Suppose $\phi_{t}$ is a Hamiltonian system with $\gamma(\phi_{t})<\hbar(L)$. Essentially by the definition of the spectral norm, there are cocycles $x\in\cf(\phi^{-1}_t;\cD)$ and $y\in\cf(\phi_t;\cD)$, representing the unit elements, with actions:
\begin{equation*}
  u=\cA_{\phi_t^{-1}}(x)=c(1,\phi_t^{-1})\quad\text{and}\quad v=\cA_{\phi_t}(y)=c(1,\phi_t)\quad\text{ and }\quad u+v=\gamma(\phi_{t}).
\end{equation*}
The idea is to prove that, for some interval $I$ of length $<\hbar(L)$, the following composition:
\begin{equation}\label{eq:kislev-shelukhin-composition}
  [\mu(-,y)]\circ\mathrm{cap}_{[\mathrm{pt}]}\circ[\mu(-,x)]: \hf(L,\id;\cD)^{I} \rightarrow\hf(L,\id;\cD)^{I+\gamma(\phi)+2\epsilon}
\end{equation}
is non-trivial, and hence $\mathrm{cap}_{[\mathrm{pt}]}:\mathrm{HF}(L,\phi_{t};\mathscr{D})^{I+u+\epsilon}\to \mathrm{HF}(L,\phi_{t};\mathscr{D})^{I+u+\epsilon}$ is non-trivial.

Note that $\mathrm{cap}_{[\mathrm{pt}]}=\mathrm{cap}_{\Pi_1}\circ\cdots\circ\mathrm{cap}_{\Pi_k}$ by the associativity of the Lagrangian cap-action, and once we know this $k$-fold composition is non-trivial, the argument proceeds exactly as in \S\ref{sec:lagr-cap-acti}. Therefore, it is enough to show that \eqref{eq:kislev-shelukhin-composition} is non-trivial in the specified action windows. To this end, we recall the associativity relation: \[[\mu(-,y)]\circ\mathrm{cap}_{[\mathrm{pt}]}\circ[\mu(-,x)]=\mathrm{cap}_{[\mathrm{pt}]}\circ[\mu(-,\mu_2(x,y))],\] from \cite[\S5.1, (20)]{kislev_shelukhin}, where $\mu_2(x,y)$ represents the unit class in $\cf(\id;\cD)$. Furthermore, $\Phi=[\mu(-,\mu_2(x,y))]$ induces the interval shift map: \[\Phi:\hf(L,\id;\cD)^{I}\rightarrow\hf(L,\id;\cD)^{I+\gamma(\phi_{t})+2\epsilon}.\] To conclude the argument, we observe that the following diagram is commutative:
\begin{equation*}\label{eq:commutative_diagram}
  \begin{tikzcd}
    \hf(L,\id;\cD)^{I} \arrow[r,] & \hf(L,\id;\cD)^{I+\gamma(\phi_{t})+2\epsilon}\\
    H(L)\arrow[u]\arrow[r,"{\cap[\mathrm{pt}]}"] & H(L)\arrow[u],
  \end{tikzcd}
\end{equation*}
where the top horizontal arrow is the composition $\mathrm{cap}_{[\mathrm{pt}]}\circ\Phi$ and the vertical arrows are the PSS morphisms. Moreover, \emph{if $I$ and $I+\gamma(\phi_{t})+2\epsilon$ both contain $0$, then both vertical maps are injective}, and hence the top map is non-trivial, as desired. Thus the problem boils down to finding an interval $I$ so that $\abs{I}<\hbar(L)$ and so that $I,I+\gamma(\phi_{t})+2\epsilon$ both contain $0$. It is clear that a necessary and sufficient condition for this to hold is that $\gamma(\phi_{t})<\hbar(L)$. This concludes the sketch of the proof of Theorem \ref{thm:main} in the closed monotone setting using the framework introduced by \cite{kislev_shelukhin}.

\subsubsection{A moduli space approach}
\label{sec:moduli-approach}

The approach to proving Theorem \ref{thm:main} adopted in this paper is to deform the pair-of-pants operation on Hamiltonian Floer cohomology using the compact Lagrangian $L$. The deformation is illustrated in Figures \ref{fig:deforming-pop} and \ref{fig:deforming-pop-2}. The details of the deformation argument are given in \S\ref{sec:deforming-pair-pants}.

The crux of the matter is to construct curves with Lagrangian boundary condition which contain conformally embedded strips with a large modulus, as in Figure \ref{fig:deforming-pop-2}. By an appropriate compactness argument, one concludes the existence of chains of Floer strips needed to prove Theorem \ref{thm:main}; the argument proceeds as in \S\ref{sec:overview}.

\begin{figure}[H]
  \centering
  \begin{tikzpicture}[yscale=.7]
    \draw (0,2) circle (0.5 and 0.1) coordinate (X1) (2,2) circle (0.5 and 0.1) coordinate(X2) (1,0) circle (0.5 and 0.1) coordinate(X3);
    \path (X1)+(0.5,0)coordinate(X1p)--+(-0.5,0)coordinate(X1m)--(X3)--+(0.5,0)coordinate(X3p)--+(-0.5,0)coordinate(X3m)--(X2)--+(0.5,0)coordinate(X2p)--+(-0.5,0)coordinate(X2m);
    \draw (X1p)to[out=-90,in=-90](X2m) (X2p)to[out=-90,in=90](X3p) (X3m)to[out=90,in=-90](X1m);
    \draw (1,-0.3) circle(0.5 and 0.1) coordinate(X) (1,-1) circle (1 and 0.2) coordinate(Y);
    \path (X)--+(-0.5,0)coordinate(Xm)--+(0.5,0)coordinate(Xp)--(Y)--+(1,0)coordinate(Yp)--+(-1,0)coordinate(Ym);
    \draw (Xm)to[out=-90,in=70] (Ym) (Xp)to[out=-90,in=110] (Yp);
    \node at (X3m) [left] {$\gamma_{\infty}$};
    \node at (X1m) [left] {$\gamma_{0}$};
    \node at (X2p) [right] {$\gamma_{1}$};
    \node at (Ym) [left] {$L$};
    \node[shift={(0,-0.14)}, fill=black,circle,inner sep=1pt] (X)at (Y){};
    \node at (X)[below]{$\mathrm{pt}$};
    \begin{scope}[shift={(5,-0.5)}]
      \draw (0,2) circle (0.3 and 0.1) coordinate (X1) (3,2) circle (0.3 and 0.1) coordinate (X2) (1.5,0) circle (2.3 and 0.3) coordinate(X3);
      \foreach \a in {1,2} {
        \path (X\a)+(0.3,0)coordinate(Xp\a)--+(-0.3,0)coordinate(Xm\a);
      }
      \path (X3)+(2.3,0)coordinate(Xp3)--+(-2.3,0)coordinate(Xm3);
      \draw (Xp1)to[out=-90,in=-90,looseness=2](Xm2) (Xm1)to[out=-90,in=60](Xm3) (Xp2)to[out=-90,in=120](Xp3);
      \node at (Xm3) [left] {$L$};
      \node at (Xm1) [left] {$\gamma_{0}$};
      \node at (Xp2) [right] {$\gamma_{1}$};
      \node[shift={(0,-0.21)}, fill=black,circle,inner sep=1pt] (X) at (X3){};
      \node at (X)[below]{$\mathrm{pt}$};
    \end{scope}
  \end{tikzpicture}
  \caption{(left) Gluing the pair-of-pants onto a half-infinite cylinder; (right) deforming the conformal structure of the resulting Riemann surface.}
  \label{fig:deforming-pop}
\end{figure}
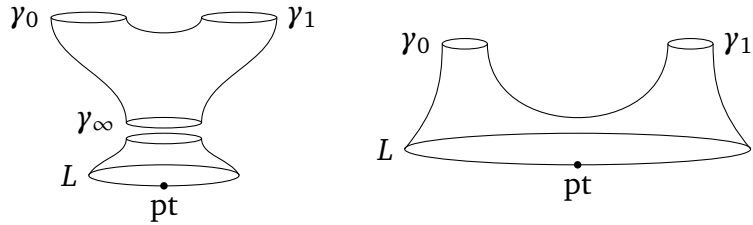

\begin{figure}[H]
  \centering
  \begin{tikzpicture}[yscale=.7]
    \draw (-2,2) circle (0.3 and 0.1) coordinate (X1) (5,2) circle (0.3 and 0.1) coordinate (X2) (1.5,0) circle (5.3 and 0.3) coordinate(X3);
    \foreach \a in {1,2} {
      \path (X\a)+(0.3,0)coordinate(Xp\a)--+(-0.3,0)coordinate(Xm\a);
    }
    \path (X3)+(5.3,0)coordinate(Xp3)--+(-5.3,0)coordinate(Xm3);
    \draw (Xp1)to[out=-90,in=-90,looseness=0.8](Xm2) (Xm1)to[out=-90,in=30](Xm3) (Xp2)to[out=-90,in=150](Xp3);
    \node at (Xm3) [left] {$L$};
    \node at (Xm1) [left] {$\gamma_{0}$};
    \node at (Xp2) [right] {$\gamma_{1}$};
    \path (X3)--+(-70:5.3 and 0.3)coordinate(S1)--+(-90:5.3 and 0.3)coordinate(S2)--+(-110:5.3 and 0.3)coordinate(S3);
    \path[every node/.style={fill,circle,inner sep=1pt}] (S1)node {} -- (S2) node{} -- (S3) node {};
    \path[every node/.style={below}] (S1)node {$f_{3}(P_{3})$} -- (S2) node{$f_{2}(P_{2})$} -- (S3) node {$f_{1}(P_{1})$};
  \end{tikzpicture}
  \caption{Deforming the conformal structure, and splitting the point constraints. The homological count of elements of such a deformed pair-of-pants will equal $1$ if the intersection of the bordism classes equals the point class $\mathrm{pt}$.}
  \label{fig:deforming-pop-2}
\end{figure}
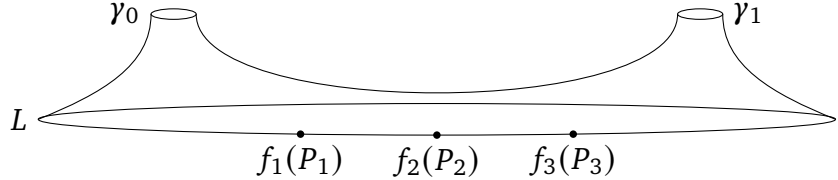

There is one subtlety in the proof of Theorem \ref{thm:main} which we explain here. The arguments in \S\ref{sec:deforming-pair-pants} imply that for each Hamiltonian system $\phi_t$ with $\phi_1=\phi$ there is an interval of length $\gamma(\phi_t)$ containing at least $\cl_{\F_2}(L)+1$ action values. It is a standard fact that if $\phi^\prime_{t}$ is another Hamiltonian system with $\phi^\prime_1=\phi$ then the action spectrums of $\phi_{t}$ and $\phi_{t}'$ coincide up to a shift by a constant depending only on the Hamiltonian loop $\phi^\prime_t\circ\phi_t^{-1}$; see, for example, \cite[Proposition 31]{kislev_shelukhin}.

Therefore, for all Hamiltonian systems $\phi_{t}$ generating $\phi$, it is possible to find a closed interval $I(\phi_{t}')$ of length $\gamma(\phi_{t}')$ containing $\cl_{\F_2}(L)+1$ action values of $\mathscr{A}_{\phi_{t}}$. The infimum of $\gamma(\phi_{t}')$ over $\phi_{t}'$ equals the spectral norm $\gamma(\phi)$. In the weakly-exact case the intervals $I(\phi_{t}')$ must remain in a fixed compact set (because $\phi_{1}(L)\cap L$ is finite). Otherwise, one can exploit the periodicity of the action spectrum and shift $I(\phi_{t}')$ to ensure the intervals do not drift off to infinity. In either case, a standard compactness argument for shrinking intervals contained in a compact set ensures the existence of an interval of length $\gamma(\phi)$ containing $\mathrm{cl}_{\F_{2}}(L)+1$ action values.

\subsection{Acknowledgements}
\label{sec:acknowledgements}

First and foremost the authors wish to thank E.~Shelukhin for suggesting this project and providing valuable guidance. The authors also wish to thank O.~Cornea, P-A.~Mailhot, and P.~Biran for clarifying discussions. This work is part of the first two authors' Ph.D.~theses under the supervision of E.~Shelukhin. The authors were supported in their research at Universit\'e de Montr\'eal by funding from the Fondation Courtois, the ISM, the FRQNT, and the Fondation J.~Armand Bombardier.

\section{Floer cohomology in convex-at-infinity symplectic manifolds}
\label{sec:floer-cohomology}

\subsection{Convex ends}
\label{sec:convex_end}

A convex end is a non-compact symplectic manifold modelled on the positive half of the symplectization of a contact manifold; see \cite{yasha-gromov}, \cite[\S11]{cieliebak_eliashberg_stein}, and \cite{ginzburg-2005-weinstein,mcduffsalamon,frauenfelder_schlenk,lanzat_quasimorphisms,lanzat_convex_HF}. Below, we define the class of symplectic manifolds called \emph{convex-at-infinity} described in \S\ref{sec:main_results}. In \S\ref{sec:structural_results}, it is shown that every convex-at-infinity manifold $W$ can be expressed as the completion of a compact symplectic manifold $\Omega$ with contact type boundary $\bd\Omega$ (allowing $\bd \Omega=\emptyset$); here \emph{contact type} is understood in the sense of \cite{weinstein-conjecture,mcduff-contact-boundaries}.

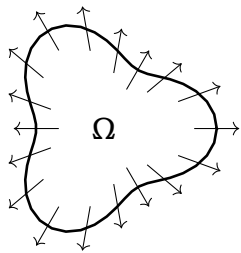
\begin{figure}[H]
  \centering
  \begin{tikzpicture}[scale=0.6]
    \node at (0,0) {$\Omega$};
    \foreach \x in {0,20,...,340} {
      \draw[->] (\x:{1.5+0.5*cos(3*\x)})--(\x:{2.5+0.5*cos(3*\x)});
    }
    \draw[line width=1pt] plot[variable=\r,domain=0:360,samples=50] ({\r}:{2+0.5*cos(3*\r)}) --cycle;
  \end{tikzpicture}
  \caption{Every convex-at-infinity symplectic manifold can be presented as the completion of a starshaped domain $\Omega$. The characteristic foliation of the boundary $\bd\Omega$ is the Reeb flow for some choice of contact form on the ideal boundary $Y$.}
  \label{fig:completion}
\end{figure}

\subsubsection{Definition of a convex end}
\label{sec:defin-conv-end}
Let $(W,\omega)$ be a symplectic manifold. Suppose there is a complete vector field $Z$ whose time $s$ flow $\rho_{s}$ satisfies the following properties:
\begin{enumerate}
\item for any sequence $z_{n}$ there is a sequence $S_{n}<0$ so that $\rho_{s_{n}}(z_{n})$ has a convergent subsequence whenever $s_{n}<S_{n}$,
\item there is a compact set $K_{1}$ so that, if $\rho_{s_{n}}(z_{n})$ and $z_{n}$ converge, and $z_{n}$ is not in $K_{1}$, then $s_{n}$ is bounded from above,
\item $(\rho_{s}^{*}\omega)_{z}=e^{s}\omega_{z}$ holds for $s>0$ and $z$ outside a compact set $K_{2}$.
\end{enumerate}
Note that if $Z_{1}=Z_{2}$ holds outside a compact set, and $Z_{1}$ satisfies the above properties, then so does $Z_{2}$, and we say $Z_{1},Z_{2}$ are equivalent. A \emph{convex end} on $W$ is an equivalence class of such vector fields $Z$. If $W$ is a compact symplectic manifold, then we can take $K_{1}=K_{2}=W$, in which case the axioms are trivially satisfied for any vector field.

The flow by $Z$ (outside $K_{2}$) will be referred to as the \emph{Liouville flow}.

\subsubsection{Structural results about convex ends}
\label{sec:structural_results}

Let $V$ be the open set of points in $W$ lying on trajectories of $Z$ which pass through $K_{1}^{c}$. Axioms (i) and (ii) ensure that $\rho_{s}$ defines a free and proper $\R$-action on $V$; the details of the argument are left to the reader.

It follows that $V\to V/\R$ is a smooth submersion to a manifold whose fibers are the trajectories of $Z$. We claim that $V/\R$ is a compact manifold. Indeed, pick a sequence $z_{n}$ in $V$ and choose a compact neighborhood $K_{1}'$ of $K_{1}$ which contains $K_{1}$ in its interior. There is $s_{n}$ so $\rho_{s_{n}}(z_{n})\in K_{1}'$, because the limit points of $\rho_{s}(z_{n})$ as $s$ converges to $-\infty$ are non-empty, by (i), and must be contained in $K_{1}$ by the properness established above. Then $\rho_{s_{n}}(z_{n})$ has a convergent subsequence since $K_{1}'$ is compact. Thus $V/\R$ is compact.

A variant of the Ehresmann fibration theorem implies that $V\to V/\R$ is a fiber bundle whose structure group is the group of translations on $\R$. Thus $V\to V/\R$ admits a section, denoted $\bd \Omega$. It follows that $Z$ is transverse to $\bd\Omega$, and the flow map takes $\bd \Omega\times [0,\infty)$ onto a neighborhood of $\infty$ for $W$. Note that $\bd\Omega$ bounds a compact domain $\Omega$ in $W$, by axiom (i).

Replacing $\bd\Omega$ by $\rho_{s}(\bd\Omega)$, we may assume from axiom (iii) that $\rho_{s}^{*}\omega=e^{s}\omega$ holds on the image of $\bd\Omega\times [0,\infty)$. The region $\bd\Omega\times [0,\infty)$ is symplectomorphic to the positive half of a symplectization of $\bd\Omega$ with contact form $\alpha=\lambda|_{\bd\Omega}$, where $\lambda=\omega(Z,-)$.  Let us refer to such a domain $\Omega$ as \emph{starshaped}.

The induced contact distribution on $V/\R$ is independent of $\Omega$ and the resulting contact manifold is called the \emph{ideal contact boundary} of $W$.

\subsubsection{Contact-at-infinity Hamiltonian systems}
\label{sec:admissible_hams}
With respect to the convex end in \S\ref{sec:convex_end}, a Hamiltonian system $\varphi_{t}$ is said to be \emph{contact-at-infinity} provided it is equivariant with respect to the positive Liouville flow, i.e., $\varphi_{t}\circ \rho_{s}(z)=\rho_{s}\circ \varphi_{t}(z)$ for $s\ge 0$, outside of a compact set. It follows that any Hamiltonian function generating $\varphi_{t}$ is one-homogeneous outside of a compact set, up to the addition of a function which is locally constant outside of a compact set; here a function $f:W\to \R$ is \emph{one-homogeneous} if $f\circ \rho_{s}(x)=e^{s}f(x)$ holds for $s\ge 0$.

Since $\varphi_{t}$ is equivariant with respect to the Liouville flow in the end $\bd\Omega\times [0,\infty)$, for suitably large star-shaped domain, the induced flow on the ideal boundary is a contact isotopy called the \emph{ideal restriction} of the system $\varphi_{t}$.

\subsubsection{Reeb flow in the convex end}
\label{sec:reeb-flow-convex-end}

A starshaped domain $\Omega$ induces a function $r$ by requiring that $r$ is $1$-homogeneous outside of $\Omega$ and $r|_{\bd\Omega}=1$. The function $r$ should be extended smoothly to all of $W$, in such a way that $r\le 1$ holds on $\Omega$.

The restriction of $\lambda$ to $\bd\Omega$ defines a contact form $\alpha$ for the ideal contact boundary. The Reeb flow for $\alpha$ is the ideal restriction of the Hamiltonian system generated by $r$.

Let $f$ be a smooth convex function so that $f(x)=x$ for $x\ge 1$, and $f(x)=0$ for $x\le 0$. Consider $f(r-r_{0})+r_{0}$ as generating an autonomous Hamiltonian system, and let $R_{t}^{\alpha}$ denote the flow induced by $X_{f(r-r_{0})+r_{0}}$. This system equals the identity on the region where $r< r_{0}$ and equals the Reeb flow for $\alpha$ on the region where $r\ge r_{0}+1$.

\subsection{Cappings and the Hamiltonian action functional}
\label{sec:hamiltonian_action}

The Hamiltonian action functional is defined on suitable a covering space of the space of contractible loops in $W$ or paths in $W$ with endpoints on the Lagrangian $L$.

\subsubsection{Cappings of chords and orbits}
\label{sec:capp-chords-orbits}

Every contractible loop or path $\gamma$ can be joined to a constant loop or path via a smooth map $u:[0,1]\times S\to W$, where $S=\R/\Z$ or $S=[0,1]$, so that (i) $u(0,t)$ is constant, (ii) $u(s,0),u(s,1)\in L$ in the case $S=[0,1]$, and (iii) $u(1,t)=\gamma(t)$. Such maps can be considered up to homotopy, and a homotopy class $[u]$ is called a \emph{capping} of $\gamma$. The projection $(\gamma,[u])\mapsto \gamma$ is a covering space.

\subsubsection{The Hamiltonian action functional}
\label{sec:hamilt-acti-funct}

Given a contact-at-infinity Hamiltonian system $\varphi_{t}$ one defines a \emph{Hamiltonian action functional} on the covering space by:
\begin{equation}\label{eq:hamiltonian_action}
  \mathscr{A}_{\varphi_{t}}(\gamma,[u])\coloneq\int_{0}^{1}H_{t}(\gamma(t))\,\d t+\mathscr{A}(\gamma,[u]),
\end{equation}
where $H_{t}$ is the unique normalized time-dependent family of smooth functions generating $\varphi_{t}$; see \S\ref{sec:norm-cond} for the normalization conditions. It is well-known, and easy to check, that the critical points of $\mathscr{A}_{\varphi_{t}}$ are lifts of contractible orbits or chords of the system $\varphi_{t}$.

If $\gamma$ is a critical point of $\mathscr{A}_{\varphi_{t}}$, then the fiber over $\gamma$ is in bijection with the set $\pi_{2}(W,\gamma(0))$ or $\pi_{2}(W,L,\gamma(0))$. The action values attained on this fiber consist, up to a constant depending on $\gamma$, of the values $\omega(\pi_{2}(W))\subset \R$ or $\omega(\pi_{2}(W,L))\subset \R$.

\subsubsection{Normalization conditions}
\label{sec:norm-cond}

Throughout this paper we assume that $W$ is connected. We do not assume that the ideal boundary $Y$ is connected, and for this reason we pick a connected component in $Y$ to be the \emph{distinguished} component. This choice is only used to define the normalization condition for Hamiltonian functions.

If $(W,\omega)$ is compact, say that $H\in C^{\infty}(W,\R)$ is \emph{normalized} if $H$ has mean zero with respect to the volume form $\omega^{n}$; if $(W,\omega)$ is non-compact (with a convex end), say that $H$ is \emph{normalized} if $H$ is one-homogeneous in the distinguished end.

Crucially: \emph{any constant normalized function is zero}.

\subsection{Hamiltonian Floer cohomology}
\label{sec:HFC-intro}
Let $(W,\omega)$ be a convex-at-infinity symplectic manifold, and let $\phi_{t}$ be a compactly supported Hamiltonian system.

To define Hamiltonian Floer cohomology for $\phi_{t}$, its necessary to deform $\phi_{t}$ on the non-compact end to make its orbits non-degenerate; we follow the approach of, for example, \cite{seidel-biased,frauenfelder_schlenk,ritter_tqft,PA_spectral_diameter} and consider systems of the form $R_{\epsilon t}^{\alpha}\circ \phi_{t}$, where $\alpha$ is a contact form on the ideal contact boundary of $W$ and $R_{s}^{\alpha}$ denotes the time $s$ Reeb flow, extended to the compact part of $W$ as in \S\ref{sec:reeb-flow-convex-end}.

Assume that $R_{\epsilon t}^{\alpha}=\id$ on the support of $\phi_{t}$ and $\phi_{t}^{-1}$, so that $\phi_{t}\circ R_{\epsilon t}^{\alpha}=R_{\epsilon t}^{\alpha}\circ \phi_{t}$. If $\epsilon>0$ is not a period of a closed Reeb orbit, then the system $R_{\epsilon t}^{\alpha}\circ \phi_{t}$ will have its fixed point contained in a compact subset of $W$. A compactly supported perturbation $\delta_{t}$ will ensure that the system $\delta_{t}\circ R_{\epsilon t}^{\alpha}\circ \phi_{t}$ has finitely many non-degenerate fixed points.

Note that the class of systems of the form $\delta_{t}\circ R_{\epsilon t}^{\alpha}\circ \phi_{t}$ is unchanged if we permute the order of the three terms $\delta_{t},R_{\epsilon t}^{\alpha},\phi_{t}$ (the perturbation term $\delta_{t}$ will change, but the class of systems is preserved).

The Floer complex for the perturbed system is defined to be the $\F_2$-vector space: $$\mathrm{CF}(\delta_{t}\circ R^{\alpha}_{\epsilon t}\circ \phi_{t},J)$$ of semi-infinite sums of capped $1$-periodic orbits of the perturbed system. The Floer differential depends on a choice of almost complex structure $J$, and counts Floer cylinders going from right-to-left as shown in Figure \ref{fig:cohomological_differential}; see \S\ref{sec:indep-choice-almost} for further discussion of the choice of almost complex structure.

The sums are semi-infinite in the following sense: for any action value $a$, there are only finitely many non-zero terms whose action is less than $a$. The cohomological differential increases action, and hence the subspace of semi-infinite sums of orbits whose action is at least $a$ is a subcomplex.

\begin{figure}[H]
  \centering
  \begin{tikzpicture}
    \draw (0,0)--node[pos=0.5,shift={(0,0.5)}]{$\bd_{s}u+J(u)(\bd_{t}u-X_{t})=0$}(8,0)arc(-90:90:{0.25 and 0.5})node[pos=0.5,right]{input, $x_{+}$}coordinate(X)--(0,1)arc(90:450:{0.25 and 0.5})node[pos=0.25,left]{output, $x_{-}$};
    \draw[dashed] (X)arc(90:270:{0.25 and 0.5});
  \end{tikzpicture}
  \caption{Cohomological convention for Floer cylinders.}
  \label{fig:cohomological_differential}
\end{figure}

The homology is denoted $\mathrm{HF}(\delta_{t}\circ R^{\alpha}_{\epsilon t}\circ \phi_{t})$. Counting continuation cylinders produces maps $\mathrm{HF}(\delta_{t}'\circ R^{\alpha}_{\epsilon' t}\circ \phi_{t})\to \mathrm{HF}(\delta_{t}\circ R^{\alpha}_{\epsilon t}\circ \phi_{t})$ provided $\epsilon'\le \epsilon$; see \S\ref{sec:continuation-maps}. One should note that continuation maps are not supposed to preserve the action filtration.

The Floer cohomology of the system $\phi_{t}$ is defined as a limit over continuation maps:
\begin{equation*}
  \mathrm{HF}(\phi_{t})=\lim_{\epsilon\to 0}\lim_{\delta\to 0} \mathrm{HF}(\delta_{t}\circ R_{\epsilon t}^{\alpha}\circ \phi_{t}).
\end{equation*}
By definition, the action of a cohomology class is given by the formula:
\begin{equation}\label{eq:min-max-action-formula}
  \mathscr{A}([x])=\sup\set{\mathscr{A}(x+\d\beta):\beta\in \mathrm{CF}(\delta_{t}\circ R_{\epsilon t}^{\alpha}\circ \phi_{t})},
\end{equation}
where $\mathscr{A}(\textstyle\sum a_{i}x_{i})=\min\set{\mathscr{A}(x_{i}):a_{i}\ne 0}$. For any element $\mathfrak{e}\in \mathrm{HF}(\phi_{t})$, one can consider its image $\mathfrak{e}_{\delta,\epsilon}\in \mathrm{HF}(\delta_{t}\circ R_{\epsilon t}^{\alpha}\circ \phi_{t})$ and take the action $\mathscr{A}(\mathfrak{e}_{\delta,\epsilon})$. In \S\ref{sec:cont-spectr-invar}, it is shown that $\mathscr{A}(\mathfrak{e}_{\delta,\epsilon})$ converges as $\delta,\epsilon$ converge to zero; we call this number the \emph{min-max action value} $\mathscr{A}_{\phi_{t}}(\mathfrak{e})$ of the class $\mathfrak{e}$.

Standard continuation arguments produce canonical isomorphisms $\mathrm{HF}(\id)\to \mathrm{HF}(\phi_{t})$ which are coherent with respect to continuation maps. If $\mathfrak{e}\in \mathrm{HF}(\id)$, then the min-max action value of the image of $\mathfrak{e}$ in $\mathrm{HF}(\phi_{t})$ is called the \emph{spectral invariant} of $\mathfrak{e}$ and is denoted $c(\mathfrak{e},\phi_{t})$.

\subsubsection{Spectral norm for a compactly supported system}
\label{sec:spectr-norm-comp}

In \S\ref{sec:pss-unit-element} the distinguished unit element $1_{\phi_{t}}\in \mathrm{HF}(\phi_{t})$ is recalled. The \emph{spectral norm} of the system is defined by the formula \eqref{eq:defin-spectral-norm} in \S\ref{sec:introduction}. It is proved in \cite{schwarz_spectral_invariants,frauenfelder_schlenk} that this depends only on the time-one map $\phi_{1}$, in the case when $W$ is aspherical. See also \cite{oh-2005-birkhauser,oh-2005-duke} for the definition of the spectral norm in the presence of holomorphic spheres.

In any case, without appealing to these results, the spectral norm of a time-1 map $\phi_{1}$ is defined to be the infimum of the spectral norms of all systems which generate $\phi_{1}$.

\subsubsection{Floer differential}
\label{sec:floer-differential}

In general, if $\varphi_{t}$ is a non-degenerate contact-at-infinity system (e.g., the system $\delta_{t}\circ R_{\epsilon t}\circ \phi_{t}$), one can form the vector space $\mathrm{CF}(\varphi_{t},J)$ of semi-infinite sums of capped contractible orbits. To define the Floer differential, one requires an almost complex structure $J$ which is tamed by $\omega$ and Liouville-equivariant outside of a compact set; see \S\ref{sec:tame-sympl-manif} for further discussion. Associated to this pair one considers the moduli space $\mathscr{M}(\varphi_{t},J)$ of Floer cylinders, as in Figure \ref{fig:cohomological_differential}.

Because we are using the semipositive framework for ensuring compactness, as in \cite{hofer-salamon-95}, we need to impose a few slightly technical conditions. First of all, we pick $J$ so that the moduli space $\mathscr{M}^{*}(A,J)$ of simple parametrized holomorphic spheres is cut transversally and the evaluation map:
\begin{equation}\label{eq:semipositive-2}
  u\in \mathscr{M}^{*}(A,J)/\mathrm{Aut}(\mathbb{C})\to u(\infty)\in W
\end{equation}
defines a pseudocycle of dimension $2n+2c_{1}(A)-4$ for every homology class $A$; see \cite[Chapter 6]{mcduffsalamon}.

Consider now the evaluation map:
\begin{equation}\label{eq:semipositive-1}
  (\mathscr{M}(\varphi_{t},J)\times \R\times \R/\Z)/\R\to W
\end{equation}
given by $(u,s,t)\mapsto u(s,t)$. We say that data $\varphi_{t}$ is \emph{admissible} for $J$ provided:
\begin{enumerate}
\item $\varphi_{1}$ is non-degenerate,
\item $\mathscr{M}(\varphi_{t},J)$ is cut transversally,
\item the evaluation map \eqref{eq:semipositive-1} is transverse to the pseudocycle defined by \eqref{eq:semipositive-2}.
\end{enumerate}

Let $\mathscr{M}_{d}(\varphi_{t},J)$ denote the $d$-dimensional component of the moduli space.

Define a differential on $\mathrm{CF}(\varphi_{t},J)$ by the formula:
\begin{equation}\label{eq:floer_differential}
  dx=\textstyle\sum_{y}n(y,x)y,
\end{equation}
where $y$ is required to have the induced capping, and:
\begin{equation*}
  n(y,x)=\#\set{u\in \mathscr{M}_{1}(\varphi_{t},J)/\R: y=u(-\infty,t)\text{ and }u(+\infty)=x}\text{ mod 2};
\end{equation*}
the quotient by $\R$ is with respect to retranslations. The sum defining $dx$ may be infinite; however, it is semi-infinite in the sense considered above, and hence is a well-defined element of the Floer complex.

As is usual in Floer theory, $d^{2}=0$ holds by considering the non-compact ends of the one-dimensional manifold $\mathscr{M}_{2}(\varphi_{t},J)/\R$; see, e.g., \cite[Theorem 4]{floer_cup}, \cite[Theorem 5.1]{hofer-salamon-95}. We briefly remark on a priori estimates needed in order for the $d^{2}=0$ argument to work. One uses bubbling analysis and the semipositivity assumption to ensure that $C^{1}$ gradient bounds follow from energy bounds. The details of the argument are given in the later sections \S\ref{sec:apriori_estimates} and \S\ref{sec:semipositivity}.

\subsubsection{Continuation maps}
\label{sec:continuation-maps}

Let $\varphi_{s,t}$ be a path of contact-at-infinity systems, satisfying (i) $\varphi_{s,0}=\id$, and (ii) $\bd_{s}\varphi_{s,t}=0$ for $s$ outside a compact interval $[s_{0},s_{1}]$. Denote by $X_{s,t}\circ \varphi_{s,t}=\bd_{t}\varphi_{s,t}$ the generating vector field.

Associated to this continuation data, define $\mathscr{M}(\varphi_{s,t},J)$ to be the moduli space of continuation cylinder $u$ solving:
\begin{equation*}
  \bd_{s}u+J(u)(\bd_{t}u-X_{s,t}(u))=0.
\end{equation*}
If $\varphi_{s,t}$ is \emph{non-positive at infinity}, in the sense that its normalized generator $H_{s,t}$ satisfies $\bd_{s}H_{s,t}\le 0$ outside of a compact set, then $\mathscr{M}(\varphi_{s,t},J)$ satisfies an a priori energy bound; the energy estimate is well-known, see \S\ref{sec:energy_estimates} for related discussion.

Let us say that $\varphi_{s,t}$ is \emph{admissible continuation data} provided:
\begin{enumerate}
\item it is non-positive at infinity, in the above sense,
\item the endpoints $\varphi_{\pm,t}$ are admissible for defining the Floer complex with $J$,
\item the moduli space $\mathscr{M}(\varphi_{s,t},J)$ is cut transversally.
\item the evaluation map:
  \begin{equation*}
    (u,s,t)\in \mathscr{M}(\varphi_{s,t},J)\mapsto u(s,t)\in W
  \end{equation*}
  is transverse to the simple $J$-spheres described in \eqref{eq:semipositive-2}.
\end{enumerate}

To such a path one associates a continuation map:
\begin{equation*}
  \mathfrak{c}:\mathrm{CF}(\varphi_{+,t},J)\to \mathrm{CF}(\varphi_{-,t},J),
\end{equation*}
by counting the rigid continuation cylinders, going from right-to-left as in \eqref{eq:floer_differential}.

Consideration of the $1$-dimensional component of $\mathscr{M}(\varphi_{s,t},J)$ proves $\mathfrak{c}$ is a chain map with respect to the Floer differentials $d_{\pm}$. The chain homotopy class of the map is unchanged under homotopies of continuation data with fixed endpoints, which are non-positive during the entire homotopy.

For details, see \cite[Theorem 4]{floer_cup}, \cite[Theorem 5.2]{hofer-salamon-95}, \cite[Lemma 6.13]{abouzaid_monograph}. See \cite{ritter_novikov}, \cite[\S2.2]{cant_sh_barcode} for discussion in the context of contact-at-infinity systems.

\subsubsection{Independence of the choice of almost complex structure}
\label{sec:indep-choice-almost}

For most of our arguments we use a fixed $\omega$-tame almost complex structure $J$; this simplifies notation while still enabling us to prove our main result. We note that the spectral norm is independent of the choice of $J$.
Indeed, for two choices of admissible complex structures $J,J'$, continuation isomorphisms can be defined between $\mathrm{CF}(\varphi_{t};J)\to \mathrm{CF}(\varphi_{t};J')$, in a way that preserves the min-max action value of the unit element.
This continuation argument is given in \cite[Theorem 5.2]{hofer-salamon-95} in the closed case. The convex-at-infinity setting does not complicate the argument; sharp energy estimates for the continuation cylinders are possible since the input and output systems coincide.

\subsubsection{PSS and the unit element}
\label{sec:pss-unit-element}

The goal in this section is to construct the ``unit element'' in $\mathrm{HF}(\varphi_{t})$ when the ideal restriction of $\varphi_{t}$ has a positive-at-infinity generating Hamiltonian. Morally, the unit element is defined by considering continuation cylinders from the identity to $\varphi_{t}$, as in Figure \ref{fig:continuation_cylinders_PSS}.

One picks a generic path $\varphi_{s,t}$ so that $\varphi_{s,t}=\varphi_{t}$ for $s\le s_{0}$ and $\varphi_{s,t}=\id$ for $s\ge s_{1}$, and so that $\bd_{s}H_{s,t}(x)\le 0$ holds for all $s$, for $x$ outside of a compact set. Associated to this is the moduli space $\mathscr{M}(\varphi_{s,t},J)$ of continuation cylinders. Each element of the moduli space has a removable singularities at the $s=+\infty$ end, and should be considered as map $u:\mathbb{C}\to W$ via the reparametrization $z=e^{-2\pi(s+it)}$.

The resulting PDE has a Fredholm linearization, and, under similar genericity conditions to avoid sphere bubbling, the count of rigid elements in $\mathscr{M}(\varphi_{s,t},J)$ defines a closed element in $1_{\varphi_{t}}\in \mathrm{HF}(\varphi_{t},J)$ called the \emph{unit}. The necessary compactness results follow from the same considerations as those in \S\ref{sec:continuation-maps}. The unit element does not depend on the particular choices (up to the addition of exact elements) for the same reason that the continuation map is well-defined up to chain homotopy.

Standard gluing arguments show that the unit elements are natural with respect to continuation maps.

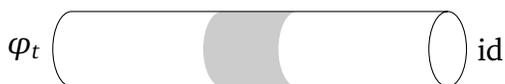
\begin{figure}[H]
  \centering
  \begin{tikzpicture}
    \path[fill=black!20!white] (3,1)--(2,1) arc (90:270:0.25 and 0.5)--(3,0) arc (-90:-270:0.25 and 0.5);
    \draw (5,1)--(0,1) arc (90:270:0.25 and 0.5) node[pos=0.5,left]{$\varphi_{t}$}coordinate(X)--(5,0) arc (-90:270:0.25 and 0.5) node[pos=0.25,right] {$\id$};
  \end{tikzpicture}
  \caption{The unit is defined via a special kind of continuation cylinder; the shaded region interpolates between $\varphi_{t}$ and $\id$.}
  \label{fig:continuation_cylinders_PSS}
\end{figure}

\subsection{pair-of-pants product}
\label{sec:pair-pants-product}

The pair-of-pants product between Floer cohomology groups in compact symplectic manifolds is well-known; see for instance \cite[\S3]{pss}, \cite[\S5.5.1.3]{schwarz-thesis}, and \cite[\S6]{seidel_representation}.

Fix the pair-of-pants $\Sigma$ to be $\mathrm{CP}^{1}$ with three punctures $0,1,\infty$; see Figure \ref{fig:pop-cp1}.
\begin{figure}[H]
  \centering
  \begin{tikzpicture}[scale=0.8]
    \path[decorate,decoration={
      markings,
      mark=between positions 0.1 and .9 step 0.2 with {\arrow{<[scale=1.3,line width=.8pt]};},
    }] (0,0) circle (0.35);
    \path[decorate,decoration={
      markings,
      mark=between positions 0.1 and .9 step 0.2 with {\arrow{<[scale=1.3,line width=.8pt]};},
    }] (1,0) circle (0.35);
    \path[decorate,decoration={
      markings,
      mark=between positions 0.05 and 1 step 0.1 with {\arrow{<[scale=1.3,line width=.8pt]};},
    }] (0,0) circle (2);
    \draw[line width=.8pt] (0,0) circle (2) (1,0) circle (0.35) node[inner sep=1pt,fill,circle]{} (0,0) circle (0.35) node[inner sep=1pt,fill,circle]{};
  \end{tikzpicture}
  \caption{pair-of-pants as $\mathrm{CP}^{1}$ with three punctures $0,1,\infty$. Circles around the punctures are oriented as the boundaries of cylindrical ends, i.e., $0,1$ are positive punctures and $\infty$ is a negative puncture.}
  \label{fig:pop-cp1}
\end{figure}
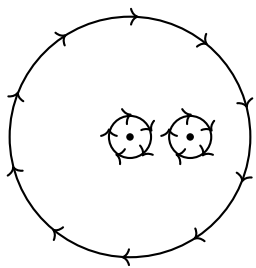

For any three systems $\varphi_{0,t},\varphi_{1,t}$ and $\varphi^{\infty}_{t}$, one can count elements of the moduli space $\mathscr{M}(\mathfrak{H},J)$ of solutions to Floer's equation where $\mathfrak{H}$ is a Hamiltonian connection over the pair-of-pants as in \S\ref{sec:floers-equat-hamilt}. See also \cite[\S8]{mcduffsalamon} for details on Hamiltonian connections and the associated moduli space $\mathscr{M}(\mathfrak{H},J)$.

The asymptotic form of $\mathfrak{H},J$ in cylindrical ends around the punctures is determined by the three systems $\varphi_{0,t},\varphi_{1,t}$ and $\varphi^{\infty}_{t}$ in such a way that solutions $u\in \mathscr{M}(\mathfrak{H},J)$ satisfy the usual Floer's equation for the $i$th system in the $i$th cylindrical end.

If $W$ is compact and semipositive, and $\mathfrak{H}$ is sufficiently generic, then the moduli space $\mathscr{M}(\mathfrak{H},J)$ has a compact zero dimensional component; the count of points in this component determines a map from $\mathrm{HF}(\varphi_{0,t})\otimes \mathrm{HF}(\varphi_{1,t})$ to $\mathrm{HF}(\varphi^{\infty}_{t})$; see \S\ref{sec:defin-pair-pants}. As explained in greater detail in \S\ref{sec:unit-times-unit}, this product sends $1_{\varphi_{0}}\otimes 1_{\varphi_{1}}$ to $1_{\varphi_{\infty}}$, i.e., it respects the unit elements coming from the $\mathrm{PSS}$ map.

The analogous product in convex-at-infinity manifolds is complicated by the need to ensure a maximum principle holds; see \cite[\S6]{ritter_tqft}, \cite{ritter_negative_line_bundles,ritter_circle_actions}, \cite[\S10.3]{abouzaid_monograph}. As explained in \S\ref{sec:maximum_principle}, a maximum principle follows from an energy bound. The results in \S\ref{sec:energy-ident-floers} prove that the energy of $u\in \mathscr{M}(\mathfrak{H},J)$ is given by a formula of the form:
\begin{equation}\label{eq:energy-ident-pop} \text{energy}(u)=\mathscr{A}_{\varphi^{\infty}_{t}}(\gamma_{\infty})-\mathscr{A}_{\varphi_{1,t}}(\gamma_{1})-\mathscr{A}_{\varphi_{0,t}}(\gamma_{0})+\int u^{*}\mathfrak{r},
\end{equation}
where $\gamma_{i}$ are the asymptotic orbits and $\mathfrak{r}$ is a ``curvature'' two-form on $\Sigma\times W$ derived from the connection $\mathfrak{H}$. In the convex-at-infinity case, it is generally impossible to bound the curvature term without suitable assumptions on the systems involved.

The strategy employed by this paper is to only consider connections $\mathfrak{H}$ which are \emph{flat}, i.e., which satisfy $\mathfrak{r}=0$. The upshot of this is that one obtains a priori energy bounds for $\mathscr{M}(\mathfrak{H},J)$ in terms of the actions of the asymptotics; see \S\ref{sec:energy_estimates} for more details.

The cost of using a flat connection is that the systems cannot be picked arbitrarily. In \S\ref{sec:locally-trivial-flat}, we explain how for any choice of contact-at-infinity systems $\varphi_{0,t},\varphi_{1,t}$, there is a flat connection $\mathfrak{H}$ on the pair-of-pants with $\varphi^{\infty}_{t}=\varphi_{0,t}\varphi_{1,t}$. In other words, if the output system is the composite of the input systems, the resulting pair-of-pants operation satisfies the energy estimate \eqref{eq:energy-ident-pop} with $\mathfrak{r}=0$; see \S\ref{sec:defin-pair-pants} for the precise statement.

Similar use of flat connections on pairs-of-pants is employed in \cite[\S4.1]{schwarz_spectral_invariants} to prove his spectral norm is sub-additive; see \S\ref{sec:sub-addit-spectr}. The pair-of-pants constructed in \cite[\S6]{ritter_tqft}, for autonomous $\varphi_{0,t}=\varphi_{1,t}$, also uses a flat Hamiltonian connection.

The paper \cite{kislev_shelukhin} introduces a general construction of flat connections on Riemann surfaces $\Sigma$ by conformally embedding strips $\R\times [0,1]$ into $\Sigma$.

As shown in Figure \ref{fig:ks-pop}, one embeds strips and requires that Floer's equation appears in the standard form on each strip; outside the strips one requires solutions are holomorphic (i.e., no Hamiltonian term). In cylindrical ends near each puncture, the conformally embedded strips are supposed to converge to strips $\R\times [t_{0},t_{1}]$ where $[t_{0},t_{1}]$ is some sub-interval of $\R/\Z$.

\begin{figure}[H]
  \centering
  \begin{tikzpicture}
    \begin{scope}[rotate=-90]
      \draw (0,2) circle (0.5 and 0.1) coordinate (X1) (2,2) circle (0.5 and 0.1) coordinate(X2) (1,-2) circle (0.5 and 0.1) coordinate(X3);
      \path (X1)+(0.5,0)coordinate(X1p)--+(-0.5,0)coordinate(X1m)--(X3)--+(0.5,0)coordinate(X3p)--+(-0.5,0)coordinate(X3m)--(X2)--+(0.5,0)coordinate(X2p)--+(-0.5,0)coordinate(X2m);
      \draw (X1p)to[out=-90,in=-90,looseness=3](X2m) (X2p)to[out=-90,in=90](X3p) (X3m)to[out=90,in=-90](X1m);
      \path (X1)+(0,.1)node[right]{input $\varphi_{0,\beta_{0}(t)}$} (X2)+(0,0.1)node[right]{input $\varphi_{0,\beta_{1}(t)}$} (X3)+(0,-0.1)node[left]{concatenated system};
      \path (X1)+(-100:0.5 and 0.1)coordinate(A1) (X3)+(130:0.5 and 0.1)coordinate(A3);
      \path (X2)+(-50:0.5 and 0.1)coordinate(B1) (X3)+(80:0.5 and 0.1)coordinate(B3);
      \draw[line width=1pt,red,fill=white!70!red] (A1)arc(-100:-130:0.5 and 0.1)to[out=-90,in=90](A3)arc(130:100:0.5 and 0.1)to[out=90,in=-90](A1);
      \draw[line width=1pt,blue,fill=white!70!blue] (B1)arc(-50:-80:0.5 and 0.1)to[out=-90,in=90](B3)arc(80:50:0.5 and 0.1)to[out=90,in=-90](B1);
    \end{scope}
  \end{tikzpicture}
  \caption{Zero curvature connections on the pair-of-pants via the strip technique of \cite{kislev_shelukhin}. Here $\beta_{i}:[0,1]\to [0,1]$ is a non-decreasing surjective smooth function which is supported in a small sub-interval.}
  \label{fig:ks-pop}
\end{figure}
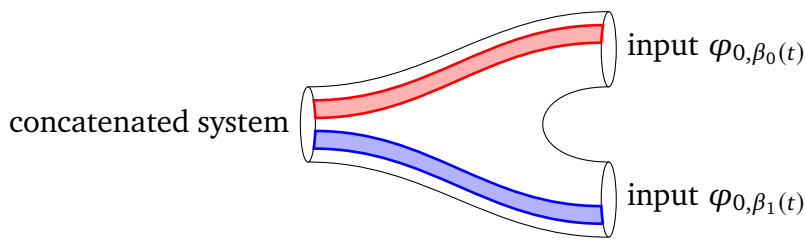

This construction can be encoded into the data $(\mathfrak{H},J)$ and $\mathfrak{H}$ is easily seen to be a flat Hamiltonian connection, by working in small enough coordinate charts where the equation appears in the above form.

The solutions to the PDE described in Figure \ref{fig:ks-pop} asymptotically satisfy Floer's equation for time-reparametrized systems near the $0$ and $1$ puncture, and satisfy Floer's equation for the appropriate concatenated system at the $\infty$ puncture.

Throughout our arguments, it is possible to use the construction of \cite{kislev_shelukhin} to produce flat Hamiltonian connections with desired asymptotic systems, alternatively, one can use the methods in \S\ref{sec:famil-flat-conn}.

\subsubsection{Definition of the pair-of-pants product}
\label{sec:defin-pair-pants}

Let $\varphi_{0,t},\varphi_{1,t},\varphi^{\infty}_{t}=\varphi_{0,t}\varphi_{1,t}$ be contact-at-infinity systems with non-degenerate fixed points. The case relevant to our paper is when the ideal restriction of $\varphi_{t}^{i}$ is $R^{\alpha}_{\epsilon t}$ and $\epsilon>0$ is smaller than the minimal $\alpha$-Reeb orbit. The systems need to be generic on the compact part of $W$ so that the Floer complexes $\mathrm{CF}(\varphi^{i}_{t})$ are well-defined; see \S\ref{sec:floer-differential}.

Let $\mathfrak{H}$ be a flat Hamiltonian connection on the pair-of-pants $\Sigma$ whose monodromies are given by $\varphi^{i}_{t}$. Abbreviate by $\mathfrak{H}+\delta$ a compactly supported perturbation of $\mathfrak{H}$.

One requires $\mathfrak{H}$ is compatible with a choice of cylindrical ends on the pair-of-pants, so that $u\in \mathscr{M}(\mathfrak{H}+\delta,J)$ solves the standard Floer's equation in the $i$th end: $$\bd_{s}u+J(u)(\bd_{t}u-X_{t}^{i}(u))=0,$$ where $X_{t}^{i}$ is the generator of $\varphi^{i}_{t}$; see \S\ref{sec:cylindrical-ends-1} for further discussion.

Associated to this data is the moduli space $\mathscr{M}(\mathfrak{H}+\delta,J)$ of finite-energy solutions to Floer's equation on $\Sigma$. We require that $\delta$ is chosen generically so that this is cut transversally and the natural evaluation map: $$\mathrm{ev}:\mathscr{M}(\mathfrak{H}+\delta,J)\to W$$ is transverse to the pseudocycle of simple $J$-holomorphic spheres; see \S\ref{sec:semipositivity}.

Similarly to the definition of the continuation map, one counts the rigid elements in $\mathscr{M}(\mathfrak{H}+\delta,J)$ as defining a map $\mathrm{HF}(\varphi_{0,t})\otimes \mathrm{HF}(\varphi_{1,t})\to \mathrm{HF}(\varphi_{\infty,t})$. Note that cappings of the asymptotic orbits at $0$ and $1$ induce a capping of the orbit at $\infty$.

The energy estimate \S\ref{eq:energy-ident-pop} implies that: $$\mathscr{A}_{\varphi_{\infty}}(\gamma_{\infty})\ge \mathscr{A}_{\varphi_{0}}(\gamma_{0})+\mathscr{A}_{\varphi_{1}}(\gamma_{1}),$$ and standard compactness results imply the sums converge (recalling that we allow semi-infinite sums). One shows via $1$-dimensional parametric moduli spaces that the resulting map is independent of the choice of perturbation $\delta$.

\subsubsection{The product of the unit with itself is the unit}
\label{sec:unit-times-unit}

The goal in this section is to prove that the unit elements constructed in \S\ref{sec:pss-unit-element} are compatible with the pair-of-pants product. This result is well-known, at least in the compact case, and we simply sketch the argument.

Perform monotone cut-offs of a flat Hamiltonian connection on the pair-of-pants in the cylindrical ends of $0$ and $1$, but do not cut off at the $\infty$ puncture. This depends on a cut-off parameter $R$, producing a connection $\mathfrak{H}(R)$ with non-zero curvature. The curvature is negative in a certain sense (this negativity would not hold if we cut off at the $\infty$ puncture). Most importantly, solutions to $\mathscr{M}(\mathfrak{H}(R),J)$ still satisfy a priori energy bounds; see \S\ref{sec:energy_estimates}.

The count of rigid elements of $\mathscr{M}(\mathfrak{H}(R),J)$ is considered as valued in $\mathrm{CF}(\varphi^{\infty}_{t})$. Standard arguments (similar to those in the construction of the unit element) show that this count defines a closed element. Deformation arguments using the parameteric construction in \S\ref{sec:famil-cylindr-ends} show that the cohomology class of this element is independent of $\varphi_{0,t},\varphi_{1,t}$, provided that $\varphi^{\infty}_{t}=\varphi_{0,t}\varphi_{1,t}$, (and also independent of $R$). Thus, we may suppose that $\varphi_{0,t}=\varphi^{\infty}_{t}$ and $\varphi_{1,t}=\id$. In this case the connection $\mathfrak{H}$ extends smoothly over the $1$ puncture. Considering $\mathrm{CP}^{1}\setminus\{0,\infty\}$ as the infinite cylinder, one shows that counting elements in $\mathscr{M}(\mathfrak{H}(R),J)$ is equivalent to counting the rigid elements of the continuation cylinders used to define the unit in \S\ref{sec:pss-unit-element}. Thus the count equals the unit.

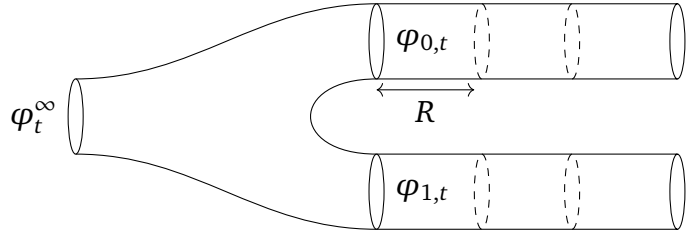
\begin{figure}[H]
  \centering
  \begin{tikzpicture}[rotate=-90]
    \draw (0,2) circle (0.5 and 0.1) coordinate (X1) (2,2) circle (0.5 and 0.1) coordinate(X2) (1,-2) circle (0.5 and 0.1) coordinate(X3);
    \path (X1)+(0.5,0)coordinate(X1p)--+(-0.5,0)coordinate(X1m)--(X3)--+(0.5,0)coordinate(X3p)--+(-0.5,0)coordinate(X3m)--(X2)--+(0.5,0)coordinate(X2p)--+(-0.5,0)coordinate(X2m);
    \draw (X1p)to[out=-90,in=-90,looseness=3](X2m) (X2p)to[out=-90,in=90](X3p) (X3m)to[out=90,in=-90](X1m);
    \path (X1)+(0,.1)node[right]{$\varphi_{0,t}$} (X2)+(0,0.1)node[right]{$\varphi_{1,t}$} (X3)+(0,-0.1)node[left]{$\varphi^{\infty}_{t}$};

    \draw (X1m)--coordinate[pos=0.35](A1)+(0,+4)arc(-180:180:0.5 and 0.1)coordinate[pos=0.5](tmp) (tmp)--coordinate[pos=0.35](B2)(X1p);
    \draw[dashed] (A1) arc (180:-180:0.5 and 0.1) (B2) arc (0:360:0.5 and 0.1);
    \path (B2)--node[pos=0.5](B2){}(X1p)node(A2){};
    \draw[<->] (B2.south)--node[below]{$R$}(A2.south);

    \draw (X2m)--coordinate[pos=0.35](A1)+(0,+4)arc(-180:180:0.5 and 0.1)coordinate[pos=0.5](tmp) (tmp)--coordinate[pos=0.35](B2)(X2p);
    \draw[dashed] (A1) arc (180:-180:0.5 and 0.1) (B2) arc (0:360:0.5 and 0.1);
  \end{tikzpicture}
  \caption{Cut-off of a flat connection on the pair-of-pants. One cuts off from $\varphi_{0,t}$ to $\id$ on an interval of fixed length, starting at a depth depending on the parameter $R$.}
  \label{fig:cut-off-flat}
\end{figure}

Now take the limit $R\to\infty$ for the original systems $\varphi_{0,t},\varphi_{1,t}$ and show that the possible breakings are configurations of a rigid pair-of-pants, with two cappings coming from the moduli spaces for the unit elements of $\varphi_{0,t}$ and $\varphi_{1,t}$. The desired result then follows, since the limiting configurations are precisely those which say that the pair-of-pants product of the two unit elements equals the unit element.

\subsubsection{Sub-additivity of spectral norm}
\label{sec:sub-addit-spectr}

Suppose that $\phi_{0,t},\phi_{1,t}$ are compactly supported Hamiltonian systems on $W$. The results in \S\ref{sec:unit-times-unit} together with the energy identity for pairs-of-pants imply that $c(1,\phi_{0,t}\circ \phi_{1,t})\ge c(1,\phi_{0,t})+c(1,\phi_{1,t})$. The spectral invariant is super-additive because we are using cohomological conventions; see \S\ref{sec:HFC-intro}.

A similar identity holds for the inverse systems, and we conclude the spectral norm from \eqref{eq:defin-spectral-norm} is sub-additive $\gamma(\phi_{0,t}\circ \phi_{1,t})\le \gamma(\phi_{0,t})+\gamma(\phi_{1,t})$.

\subsection{Deforming the pair-of-pants product using a Lagrangian}
\label{sec:deforming-pair-pants}

In this section we explain how to deform the pair-of-pants product using a compact Lagrangian. To begin, we consider the following operation defined by counting half-infinite cylinders.

If $\varphi_{\infty,t}$ is a small perturbation of the identity, with a small positive slope, then one can count half-infinite cylinders $[0,\infty)\times \R/\Z$ with Lagrangian boundary conditions solving Floer's equation for $\varphi_{\infty,t}$. This defines an augmentation type map $\mathrm{HF}(\varphi_{\infty,t})\to \Lambda.$

Here, the Novikov ring $\Lambda$ is the $\F_2$-algebra consisting of all semi-infinite sums $\sum x_{k}e^{A_{k}}$ where $A_{k}$ are areas of elements of $\pi_{2}(W,L)$; the sums are semi-infinite in the sense that only finitely many $k$ have $\omega(A_{k})\le \lambda$. We are primarily interested in the one-dimensional subspaces $\Lambda_{A}=\F_2\cdot e^{A}$ corresponding to the disk of area $A$. By construction, the subspaces: $$\Lambda_{< A}=\textstyle\bigoplus_{A'< A} \Lambda_{A'}$$ define a filtration of $\Lambda$.

The energy of a solution $u$ to the equation defined in Figure \ref{fig:augmentation} is given as:
\begin{equation}\label{eq:energy-augmentation}
  E(u)=\int_{0}^{1} H_{t}(u(0,t))\d t-\mathscr{A}_{\varphi_{\infty}}(\gamma_{\infty})+A,
\end{equation}
where $A$ is area of the element of $\pi_{2}(W,L)$ formed by gluing the cylinder to the chosen capping of $\gamma_{\infty}$. To obtain the relevant energy bound, we only define:
\begin{equation}\label{eq:augmentation}
  \mathfrak{a}:\mathrm{HF}_{>c+A-\hbar}(\varphi_{\infty,t})\to \Lambda_{<A},
\end{equation}
where $c$ is a small number defined precisely in \eqref{eq:defin-c-constant}, and $E(u)<\hbar$ is sufficient to preclude disk bubbling. To be precise, in this section we fix $\hbar$ so that:
\begin{equation}\label{eq:defn-hbar}
  \gamma(\phi_{t})<\hbar<\hbar(L),
\end{equation}
where $\hbar(L)$ is defined in \eqref{eq:hbar-constant}. Briefly, one only counts the solutions to Figure \ref{fig:augmentation} where the symplectic area of the disk formed is less than $A$ to define \eqref{eq:augmentation}; see \S\ref{sec:augm-assoc-lagr} for further details.

A special role is played by $\Lambda_{0}$, as it encodes the area of the contractible disk. Indeed, in \S\ref{sec:augm-assoc-lagr}, it is shown that the augmentation $\mathrm{HF}_{>c-\hbar}(\varphi_{\infty,t})\to \Lambda_{0}$ sends the unit element to a non-zero element, provided $\varphi_{\infty,t}$ is a sufficiently small perturbation of the identity.

\begin{figure}[H]
  \centering
  \begin{tikzpicture}
    \path (0,0) circle (0.1 and 0.5) coordinate(X) (8,0) circle (0.1 and 0.5) coordinate(Y);
    \draw (X)+(0,0.5)coordinate(X1p)arc(90:450:0.1 and 0.5)coordinate[pos=0.5](X1m) coordinate[pos=0.75](X1z);
    \draw[dashed] (Y)+(0,0.5)coordinate(Y1p)arc(90:450:0.1 and 0.5)coordinate[pos=0.5](Y1m);
    \draw (X1p)--(Y1p) (X1m)--(Y1m);
    \path (X1p)--(Y1m)node[pos=0.5]{$\bd_{s}u-J(u)(\bd_{t}u-X_{t})=0$};
    \node at (X.west)[shift={(-0.1,0)},left]{$L$};
    \node at (Y.east)[shift={(0.1,0)},right]{$\gamma_{\infty}$};
    \node[draw,inner sep=1pt,circle,fill=black] (X) at (X1z){};
    \node at(X)[right]{$\mathrm{pt}$};
  \end{tikzpicture}
  \caption{The augmentation $\mathrm{HF}(\varphi_{\infty,t})\to \Lambda$ associated to a Lagrangian $L$ with a point constraint $u(0,0)=\mathrm{pt}$.}
  \label{fig:augmentation}
\end{figure}

Now the pair-of-pants product enters the discussion.

Let $\varphi_{0,t}=\delta_{0,t}\circ R_{\epsilon t}\circ \phi_{t}^{-1}$ and $\varphi_{1,t}=\phi_{t}\circ R_{\epsilon t}\circ \delta_{1,t}$ be small perturbations of a compactly supported Hamiltonian isotopy $\phi_{t}$ and its inverse, as in \S\ref{sec:HFC-intro}, and let $\varphi_{\infty,t}=\varphi_{0,t}\varphi_{1,t}$. Then $\varphi_{\infty,t}$ is a small enough perturbation of the identity so that the augmentation $\mathrm{HF}_{>c-\hbar}(\varphi_{\infty,t})\to \Lambda_{0}$ sends the unit element to a non-zero element.
The idea is to precompose this augmentation with the pair-of-pants product: $$\mathrm{HF}_{>a}(\varphi_{0,t})\otimes \mathrm{HF}_{>b}(\varphi_{1,t}) \to \mathrm{HF}_{>c-\hbar}(\varphi_{\infty,t})$$ for suitable negative constants $a,b$; see \S\ref{sec:relev-spectr-norm} for discussion of the precise choice of constants $a,b$ and the relevance of the spectral norm.

Gluing the pair-of-pants to the half-infinite cylinder provides a chain level description for the operation:
\begin{equation}\label{eq:bi-augmentation}
  \mathrm{HF}_{>a}(\varphi_{0,t})\otimes \mathrm{HF}_{>b}(\varphi_{1,t})\to \Lambda_{0}.
\end{equation}
The operation \eqref{eq:bi-augmentation} can be defined in many ways, as explained in \S\ref{sec:prod-oper-with}.
The idea is to use the gluing trick to show the operation is non-zero (by considering it as a composition), but then deform the PDE in the manner described in Figure \ref{fig:deforming-pop}.
Via a sufficiently large deformation, one concludes a sequence of Floer strips with large modulus on $L$ and point constraints as shown in Figure \ref{fig:deforming-pop-2}; see \S\ref{sec:floer-strips-with}.
As explained in \S\ref{sec:overview}, one can conclude multiple chords with endpoints on $L$, with varying actions, thereby completing the proof of Theorem \ref{thm:main}.

The rest of this section is concerned with the technical details of the argument.

\subsubsection{Augmentation associated to a Lagrangian with a point constraint}
\label{sec:augm-assoc-lagr}

In this section we are concerned with the construction of the augmentation described above, in the case when:
\begin{equation*}
  \varphi_{\infty,t}:=\delta_{0,t}R_{2\epsilon t}^{\alpha}\delta_{1,t},
\end{equation*}
where $\delta_{i,t}$ are $C^{\infty}$ small perturbations and $\epsilon$ is a small enough positive number.

Let $\Omega(r_{0})=\set{r\le r_{0}}$ and $R^{\alpha}_{s}$, be as in \S\ref{sec:reeb-flow-convex-end}. Suppose that $r_{0}$ is large enough that $L$ is contained in the interior of $\Omega(r_{0})$.
We also assume that the perturbations $\delta_{i,t}$ are such that the normalized Hamiltonians $\Delta_{i,t}$ which generate $\delta_{i,t}$ vanish outside of $\Omega(r_{0}+1)$; this is sufficient to ensure transversality of all non-constant Floer cylinders; see \S\ref{sec:cont-spectr-invar} for further discussion.

The normalized Hamiltonian generating $\varphi_{\infty,t}$ is given by:
\begin{equation*}
  H_{t}:=\Delta_{0,t}+\Delta_{1,t}\circ R^{\alpha}_{-2\epsilon t}\delta_{0,t}^{-1}+2\epsilon f(r\circ \delta_{0,t}^{-1}-r_{0})+2\epsilon r_{0};
\end{equation*}
see \S\ref{sec:estim-error-term} for further details. By picking $\epsilon$ and $\Delta_{0,t},\Delta_{1,t}$ sufficiently small we may assume that:
\begin{equation}\label{eq:defin-c-constant}
  \int_{0}^{1}\max_{x\in \Omega(r_{0}+1)} H_{t}(x)\d t=:c<\hbar,
\end{equation}
As explained above, this is sufficient to obtain an a priori energy bound for defining the augmentation $\mathrm{HF}_{>c-\hbar}(\varphi_{\infty,t})\to \Lambda_{0}$.

The augmentation map $\mathfrak{a}:\mathrm{HF}_{>c-\hbar}(\varphi_{\infty,t})\to \Lambda_{0}$ is defined on capped orbits by the formula:
\begin{equation*}
  \mathfrak{a}(\gamma_{\infty})=\sum n(\gamma_{\infty},A)e^{A},
\end{equation*}
where $n(\gamma_{\infty},A)$ is the number of rigid half-infinite cylinders $u$ as in Figure \ref{fig:augmentation} such that $u\#\bar{\gamma}_{\infty}$ has symplectic area $A\le 0$, where $\bar{\gamma}_{\infty}$ denotes the capping. The resulting Floer energy of the cylinder is given by \eqref{eq:energy-augmentation} which is bounded by $\hbar$. Standard arguments imply that this is a chain map, where $\Lambda_{0}$ is considered as a complex with zero differential.

To prove that the unit element is sent to $1\in \Lambda_{0}$, we first need to show that the unit element $1_{\varphi_{\infty,t}}$ is represented in the $\mathrm{HF}_{>c-\hbar}(\varphi_{\infty,t})$ for $\Delta_{i,t}$ and $\epsilon$ sufficiently small. One shows this by analyzing the action of the output of the PSS map.

More precisely, we claim that, for the fixed negative error $c-\hbar$, we can pick $\Delta_{i,t},\epsilon$ sufficiently small so that the unit element $1_{\varphi_{\infty,t}}$ has action value greater than $c-\hbar$. Indeed, on chain level, $1_{\varphi_{\infty,t}}$ is defined by counting the solutions to the PSS equation described in \S\ref{sec:pss-unit-element}; one can therefore estimate the minimal action of the output:
\begin{equation*}
\text{min action}\ge \mathrm{min}_{W\times \R/\Z}H_{t}+\text{symplectic area of the PSS cylinder in Figure \ref{fig:continuation_cylinders_PSS}}.
\end{equation*}
The first term on the right converges to zero. If the sum is not greater than $c-\hbar<0$, then the symplectic area must be uniformly negative. Then we eventually contradict the non-negativity of the energy of the PSS cylinder, because:
\begin{equation*}
  0\le \text{energy}\le \int_{\gamma_{-}}H_{t}\d t+\int_{\R}\max_{W}\bd_{s}H_{s,t}+\text{symplectic area},
\end{equation*}
and the first two terms cannot be uniformly positive (as $\Delta_{i,t}$ and $\epsilon$ converge to zero), and so the last term cannot be uniformly negative.

Having established that $1_{\varphi_{\infty,t}}$ lives in the appropriate action window, it remains to show that $\mathfrak{a}$ sends $1_{\varphi_{\infty,t}}$ to $1\in \Lambda_{0}$. To show this, one considers the moduli space $\mathscr{M}(\sigma,\varphi_{\infty,t})$ of cylinders described in Figure \ref{fig:interp_augmentation}.

\begin{figure}[H]
  \centering
  \begin{tikzpicture}
    \path (0,0) circle (0.1 and 0.5) coordinate(X) (12,0) circle (0.1 and 0.5) coordinate(Y);
    \draw (X)+(0,0.5)coordinate(X1p)arc(90:450:0.1 and 0.5)coordinate[pos=0.5](X1m) coordinate[pos=0.75](X1z);
    \draw[dashed] (Y)+(0,0.5)coordinate(Y1p)arc(90:450:0.1 and 0.5)coordinate[pos=0.5](Y1m);
    \draw (X1p)--(Y1p) (X1m)--(Y1m);
    \path (X1p)--(Y1m)node[pos=0.5]{$\bd_{s}u-J(u)(\bd_{t}u-\beta(\sigma-s)X_{t})=0$};
    \node at (X.west)[shift={(-0.1,0)},left]{$L$};
    \node at (Y.east)[shift={(0.1,0)},right]{$\gamma_{\infty}$};
    \node[draw,inner sep=1pt,circle,fill=black] (X) at (X1z){};
    \node at(X)[right]{$\mathrm{pt}$};
  \end{tikzpicture}
  \caption{Here $\beta(x)=1$ for $x\ge 1$ and $\beta(x)=0$ for $x\le 0$. The cylinders interpolate between the augmentation and the PSS equation from Figure \ref{fig:continuation_cylinders_PSS}. Taking a limit $\sigma\to\infty$ shows that the augmentation sends the unit to $1\in \Lambda_{0}$.}
  \label{fig:interp_augmentation}
\end{figure}

Combining the energy estimate for the augmentation cylinders with the energy estimate for the PSS cylinders shows that, if $\Delta_{i,t}$ and $\epsilon$ are small enough, then the Floer energy of the cylinders in Figure \ref{fig:interp_augmentation} is bounded above by $\hbar+A$, here $A$ is the symplectic area of the $u$ (considered as a disk). The different symplectic areas split $\mathscr{M}(\sigma,\varphi_{\infty,t})$ into connected components.

The augmentation to $\Lambda_{0}$ only considers the components where $A=0$, and so we restrict our attention to this part of the moduli space. Then, if $\sigma$ is suficiently negative, the only elements of $\mathscr{M}(\sigma,\varphi_{\infty,t})$ are the constant $J$-holomorphic disks. Arguing similarly to \cite[\S9.2]{mcduffsalamon}, one thinks of the union of the index $0$ components $\mathscr{M}_{\mathrm{para}}=\bigcup_{\sigma}\mathscr{M}(\sigma,\varphi_{\infty,t})$ as a parametric moduli space; since there is a unique constant $J$-holomorphic disk satisfying the point constraint, the fibers over $\sigma<0$ consists of a single point. It follows that, for generic compactly-supported perturbations of the parametric equation, $\mathscr{M}_{\mathrm{para}}\to \R$ is smooth map between one-manifolds, and the fibers over generic $\sigma>0$ represent the non-trivial class in the unoriented bordism group of $0$-dimensional manifolds; in other words, the $\F_2$-valued counts of the zero symplectic area component of $\mathscr{M}(\sigma_{k},\varphi_{\infty,t})$ is exactly $1$ for generic $\sigma$.

Taking the limit $\sigma\to\infty$, and applying standard compactness and gluing arguments, one concludes that $\mathfrak{a}$ takes the unit element in $\mathrm{HF}_{>c-\hbar}(\varphi_{\infty,t})$ to $1$.

\subsubsection{Relevance of the spectral norm condition}
\label{sec:relev-spectr-norm}
The pair-of-pants operation defined in \S\ref{sec:pair-pants-product} respects the action filtrations and induces a morphism:
\begin{equation}\label{eq:filtered-pop}
  \mathrm{HF}_{>a}(\varphi_{0,t})\otimes \mathrm{HF}_{>b}(\varphi_{1,t})\to \mathrm{HF}_{>a+b-\rho_{1}}(\varphi_{\infty,t});
\end{equation}
the key is the energy estimate $E\le \mathscr{A}(\gamma_{\infty})-\mathscr{A}(\gamma_{0})-\mathscr{A}(\gamma_{1})+\rho_{1}$, where $\rho_{1}>0$ is a small error term due to perturbing a flat connection; see \S\ref{sec:energy_estimates}.

Thus, if $1_{\varphi_{0,t}}\in \mathrm{HF}_{>a}(\varphi_{0,t})$ and $1_{\varphi_{1,t}}\in \mathrm{HF}_{>b}(\varphi_{1,t})$, and $a+b-\rho_{1}>c-\hbar$, composing \eqref{eq:filtered-pop} with the augmentation $\mathfrak{a}:\mathrm{HF}_{>c-\hbar}(\varphi_{\infty,t})\to \Lambda_{0}$ yields a \emph{non-trivial} product:
\begin{equation}\label{eq:non-trivial-product}
  \mathrm{HF}_{>a}(\varphi_{0,t})\otimes \mathrm{HF}_{>b}(\varphi_{1,t})\to \Lambda_{0}
\end{equation}
sending $1_{\varphi_{0,t}}\otimes 1_{\varphi_{1,t}}$ to $1\in \Lambda_{0}$. See Figure \ref{fig:deforming-pop} for an illustration of this composition. Note that the product of the units is sent to the unit, \S\ref{sec:unit-times-unit}, but in order to preclude disk bubbling, one has to restrict the action windows, thus one needs to show that the unit in the target is in the image of the appropriate filtered complex in the domain.

If $\varphi_{0,t}=\delta_{0,t}\circ R_{\epsilon t}\circ \phi^{-1}_{t}$ and $\varphi_{1,t}=\phi_{t}\circ R_{\epsilon t}\circ \delta_{1,t}$, we can pick the perturbation terms small enough that:
\begin{equation*}
  \mathscr{A}(1_{\varphi_{0,t}})+\mathscr{A}(1_{\varphi_{1,t}})\ge c(1,\phi_{t}^{-1})+c(1,\phi_{t})-\rho_{2}=-\gamma(\phi_{t})-\rho_{2},
\end{equation*}
for an arbitrarily small $\rho_{2}>0;$ this approximation follows from the results in \S\ref{sec:cont-spectr-invar}.

Since $\gamma(\phi_{t})<\hbar$, we can pick the perturbations so that $c+\rho_{1}+\rho_{2}<\hbar-\gamma(\phi_{t})$, recalling that $c$ is defined by \eqref{eq:defin-c-constant}. Then $\mathscr{A}(1_{\varphi_{0},t})+\mathscr{A}(1_{\varphi_{1,t}})>c-\hbar$ and so \eqref{eq:non-trivial-product} is non-trivial.

\subsubsection{Product operation with multiple point constraints}
\label{sec:prod-oper-with}

We continue with the set-up of \S\ref{sec:relev-spectr-norm}, and interpret the non-trivial operation \eqref{eq:non-trivial-product} as a count of solutions to Floer's equation associated to a Hamiltonian connection over a twice-punctured disk; see Figure \ref{fig:deforming-pop} and \ref{fig:deforming-pop-2}.

Consider the closed unit disk $\Sigma_{0}$ with holomorphic coordinate $x+iy$ with two punctures at $x=\pm 0.5$. By the results in \S\ref{sec:famil-flat-conn}, on $\Sigma_{0}\times W\to \Sigma_{0}$ there is a flat Hamiltonian connection $\mathfrak{H}_{0}$ whose monodromies around $-0.5$ and $+0.5$ are $\varphi_{0,t}$ and $\varphi_{1,t}$, respectively.

Moreover, as explained in \S\ref{sec:locally-trivial-flat}, we can ensure that all monodromy maps are valued in the subgroup $\mathrm{Ham}(W;R^{\alpha},r_{0}+1)$ of those diffeomorphisms which agree with the Reeb flow outside of $\Omega(r_{0}+1)$.

Let $i_{\sigma}:D\to D$, $\sigma\in \R$, be an isotopy so that $i_{\sigma}(\pm 0.5)=\pm(0.5+\pi^{-1}\arctan(\sigma))$; in words, $i_{\sigma}$ expands the distance between the punctures as $\sigma$ increases. Then let $\Sigma_{\sigma}=i_{\sigma}(\Sigma_{0})$ and let $\mathfrak{H}^{\mathrm{tmp}}_{\sigma}=i_{\sigma,*}\mathfrak{H}$; we consider $\Sigma_{\sigma}$ as having the standard complex structure. One can think of $(\Sigma_{\sigma},\mathfrak{H}_{\sigma}^{\mathrm{tmp}},j)$ as being isomorphic to $(\Sigma_{0},\mathfrak{H}_{0},i_{\sigma}^{*}j)$.

As we want $\mathfrak{H}_{\sigma}$ to appear in a standard form on certain subsets, we will correct $\mathfrak{H}^{\mathrm{tmp}}_{\sigma}$ by fiberwise Hamiltonian diffeomorphisms of $\Sigma_{\sigma}\times W$, using the coordinate transformations introduced in \S\ref{sec:coordinate-changes-hamilt-conn}, \S\ref{sec:conj-monodr-coord}, \S\ref{sec:cylindrical-ends}, and \S\ref{sec:coord-changes-energy}. We require the following:
\begin{enumerate}
\item\label{item:sc-1} As $\sigma\to -\infty$, the connection equals the one obtained by gluing a pair-of-pants to an augmentation half cylinder, as shown on the left hand side of Figure \ref{fig:deforming-pop}.
\item\label{item:sc-2} As $\sigma\to\infty$, there is a conformally embedded neck of large modulus $N(\sigma)\subset \Sigma_{\sigma}$, as proved in \S\ref{sec:conf-embedd-strip-disk}, and we require that the connection appears in the standard form for Floer's equation on the strip for the system $\varphi_{1,t}$.
\item\label{item:sc-3} The energy estimates $E(u) \leq c-\mathscr{A}_{\varphi_{0,t}}(\gamma_{0})-\mathscr{A}_{\varphi_{1,t}}(\gamma_{1})+A$ holds for solutions $u$ with asymptotics $\gamma_{0},\gamma_{1}$, so that the disk formed has symplectic area $A$, and where $c$ is given in \eqref{eq:defin-c-constant}.
\end{enumerate}

The first property is straightforward to achieve using the connectivity of the space of flat connections with fixed monodromy representation explained in \S\ref{sec:connectivity-space-flat}.

To achieve \ref{item:sc-2}, we argue as follows. First pick small geodesic disks around the punctures which are disjoint from a holomorphic collar $C_{\infty,\sigma}$ around the boundary circle. Via a coordinate tranformation as in \S\ref{sec:cylindrical-ends}, one ensures the connection appears as the standard connection Floer's equation for $\varphi_{\infty,\kappa(t)}$, where $\kappa$ is a cut-off function $[0,1]\to [0,1]$ supported in a small interval. This is possible because $\varphi_{\infty,t}$ and $\varphi_{\infty,\kappa(t)}$ have the same time one map in the universal cover; see Proposition \ref{prop:coordinate-change-cyl}. The relevant parts of the domain are illustrated in Figure \ref{fig:constr-flat-conn}. It is crucial to note that this time reparametrization does not change the value of $c$ given in \eqref{eq:defin-c-constant}.

By picking $\kappa$ to smoothly depend on $\sigma$, we can suppose that the monodromy diffeomorphisms along arcs in the boundary, which are disjoint from the support shown in Figure \ref{fig:constr-flat-conn}, are the identity.

This construction ensures that the time-$1$ map of monodromy along the transverse arcs of the neck $N(\sigma)$ equals $\psi^{-1}\varphi_{1}\psi$. Since $\psi$ is in $\mathrm{Ham}(W;R^{\alpha},r_{0}+1)$, we can correct the connection by pushing forward by $\psi$ to ensure the monodromy along the transverse arcs equal $\varphi_{1}$. Indeed, we can push forward by a family $g$ so that $g_{z}=\psi$ holds for $z\in \bd\Sigma$, and $g_{z}=\id$ near the punctures. This achieves \ref{item:sc-2}.

The final step (pushing forward by a coordinate change $g$) can ruin the a priori energy bounds. Indeed, if $u$ solves Floer's equation for $g_{*}\mathfrak{H},J,L$ then $g^{-1}(u)$ solves Floer's equation for $\mathfrak{H},g^{-1}_{*}(J),g^{-1}(L)$; see \S\ref{sec:coord-changes-energy}. However, since $g$ is valued in $\mathrm{Ham}(W;R^{\alpha},r_{0}+1)$, $g^{-1}(L)$ still lies in $\Omega(r_{0}+1)$, and hence the constant $c$ from \eqref{eq:defin-c-constant} is unchanged when we do this conjugation. It follows that solutions $u$ in the moduli space $\mathscr{M}(\psi_{*}\mathfrak{H},J,L)$ still satisfy the energy estimate of \ref{item:sc-3}.

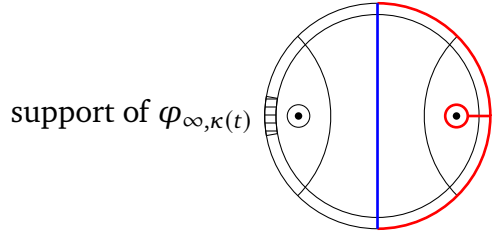
\begin{figure}[H]
  \centering
  \begin{tikzpicture}[scale=1.5]
    \begin{scope}[scale=0.5]
      \draw (-1.41421356,-1.41421356) arc (315: 405 : 2);
      \draw (1.41421356,1.41421356) arc (135: 225 : 2);
    \end{scope}
    \draw (0,0) circle (1);
    \draw (-0.7,0) circle (0.1) node[fill,circle,inner sep=1pt]{} (0.7,0) circle (0.1) node[fill,circle,inner sep=1pt]{} (0,0) circle (0.9);
    \draw[pattern={Lines[angle=40]}] (170:0.9) -- (170:1) arc (170:190:1) node[pos=0.5,left]{support of $\varphi_{\infty,\kappa(t)}$} --(190:0.9) arc (190:170:0.9);

    \draw[line width=1pt,red] (0,-1) arc (-90:0:1) -- (0.8,0) (1,0) arc (0:90:1) (0.7,0) circle (0.1);
    \draw[line width=1pt,blue] (0,-1) --(0,1);
    
  \end{tikzpicture}
  \caption{Construction of a flat connection with special properties; shown are the two cylindrical ends, the collar around the boundary, and the holomorphic neck $N(\sigma)$ of modulus $s(\sigma)\to\infty$. The coloured paths are used to compare monodromies.}
  \label{fig:constr-flat-conn}
\end{figure}

Let $\mathfrak{H}_{\sigma}$ denote the smooth family of connections obtained from $\mathfrak{H}^{\mathrm{tmp}}_{\sigma}$ by applying the coordinate changes described above. The upshot of these slightly technical requirements is that, as $\sigma\to +\infty$, solutions to $\mathscr{M}(\Sigma_{\sigma},\mathfrak{H}_{\sigma},J)$ restrict to Floer strips for the system $\varphi_{1,t}$ on the neck $N(\sigma)$.\footnote{One should also note that one can use the construction of \cite{kislev_shelukhin} described in \S\ref{sec:pair-pants-product} to construct the desired family of flat connections.}

Let $z_{1},\dots,z_{k}$ be smooth maps $\R\to \bd D$ so that $z_{1}(\sigma)=\dots=z_{k}(\sigma)=-i$ for $\sigma$ sufficiently negative. Fix also smooth maps $f_{i}:P_{i}\to L$ representing bordism classes of maps in $L$.

Consider $\mathscr{M}_{\mathrm{para}}$ to be the parametric moduli space of pairs $(p_{1},\dots,p_{k},\sigma,u)$ where $u$ is a finite energy solution to Floer's equation for $(\Sigma_{\sigma},\mathfrak{H}_{\sigma}+\delta_{\sigma},J)$ and so:
\begin{equation}\label{eq:incidence}
  u(z_{i}(\sigma))=f_{i}(p_{i}).
\end{equation}
Here $\delta_{\sigma}$ is a small uniformly compactly supported perturbation used to ensure $\mathscr{M}_{\mathrm{para}}$ is cut transversally.

Each curve appearing in $\mathscr{M}_{\mathrm{para}}$ is asymptotic at the $0$-labeled puncture to a $1$-periodic orbit $\gamma_{0}$ of $\varphi_{0,t}$ and is asymptotic at $1$-labeled puncture to a $1$-periodic orbit $\gamma_{1}$ of $\varphi_{1,t}$.

Let $c$ be given by \eqref{eq:defin-c-constant}, and restrict $\mathscr{M}_{\mathrm{para}}$ to only those $u$ whose asymptotic orbits have cappings such that the disk formed by gluing $u$ to the two cappings has zero symplectic area, and the sum of the actions of the capped orbits is greater than $c-\hbar$. It follows that the energy of any solution in $\mathscr{M}_{\mathrm{para}}$ is bounded from above by $\hbar$; in particular, no disk (or sphere) bubbling can occur in $\mathscr{M}_{\mathrm{para}}$.

Let $\mathscr{M}_{\mathrm{para},1}$ be the one-dimensional component, and consider the projection map:
\begin{equation}\label{eq:projection-map}
  (p_{1},\dots,p_{k},\sigma,u)\in \mathscr{M}_{\mathrm{para},1}\mapsto \sigma\in \R,
\end{equation}
whose fiber over $\sigma$ is denoted by $\mathscr{M}(\sigma)$. For generic $\sigma\in \R$, $\mathscr{M}(\sigma)$ is a $0$-dimensional manifold.

Suppose that $\varphi_{0,t},\varphi_{1,t}$ are as in \S\ref{sec:relev-spectr-norm}, so that the actions of the unit elements $1_{\varphi_{0,t}}$ and $1_{\varphi_{1,t}}$ are greater than $a,b$, respectively, and $a+b>c-\hbar$. In this setting, we can interpret the count of elements in $\mathscr{M}(\sigma)$ (for generic $\sigma$) as defining a map:
\begin{equation}\label{eq:b-sigma-map}
  \mathfrak{b}_{\sigma}:\mathrm{CF}(\varphi_{0,t})_{>a}\otimes \mathrm{CF}(\varphi_{1,t})_{>b}\to \Lambda_{0}.
\end{equation}
One sends a generator $\gamma_{0}\otimes \gamma_{1}$ to the number of elements in $\mathscr{M}(\sigma)$ whose asymptotics equal $\gamma_{0},\gamma_{1}$ and which form a contractible disk when glued to the cappings.

\begin{claim}\label{claim:b-sigma}
  The map $\mathfrak{b}_{\sigma}$ is a chain map for all generic $\sigma$. If the bordism classes represented by $f_{1},\dots,f_{k}$ have a total intersection product equal to the point class in $L$, then $\mathfrak{b}_{\sigma}$ is non-trivial on homology.
\end{claim}
\begin{proof}
  That $\mathfrak{b}_{\sigma}$ is a chain map follows from the usual Floer theoretic arguments.
  
  When $\sigma$ is very negative, the two punctures become arbitrarily close to each other, and $\Sigma_{\sigma}$ is conformally equivalent to a very large disk $D(R)$ with $0$ and $1$ as the punctures. Gluing arguments are then used to show that $\mathfrak{b}_{\sigma}$ equals the composition of $\mathfrak{a}$ with the pair-of-pants product as in \S\ref{sec:relev-spectr-norm}, and is therefore non-trivial on homology.

  Finally, standard deformation arguments show that $\mathfrak{b}_{\sigma}$ and $\mathfrak{b}_{\sigma'}$ are chain homotopic for $\sigma,\sigma'$ generic. In particular, the non-triviality for very negative values of $\sigma$ implies the non-triviality for all generic values of $\sigma$.
\end{proof}
In the next subsection we exploit the non-triviality of $\mathfrak{b}_{\sigma}$ for large values of $\sigma$, for specific choice of the punctures $z_{i}(\sigma)$, so as to conclude chains of non-stationary Floer strips, as described in \S\ref{sec:overview}.

\subsubsection{Floer strips with multiple point constraints}
\label{sec:floer-strips-with}

The deformation argument in \S\ref{sec:prod-oper-with} produces Floer strips of arbitrarily large modulus for the system $\varphi_{1,t}$; moreover, by picking the punctures $z_{i}(\sigma)$ to be equally spaced along the lower boundary of the Floer strip, one concludes $k$ sub-strips of large modulus, each with an incidence constraint at the midpoint of its lower boundary; see Figure \ref{fig:substrips}, and compare with Figure \ref{fig:deforming-pop-2}.

\begin{figure}[H]
  \centering
  \begin{tikzpicture}
    \draw (0,0) rectangle (10,1);
    \path[decorate,decoration={markings,mark=between positions 0.1 and 0.9 step 0.2 with {\node[fill,circle,inner sep=1pt]{};}}] (0,0)--(10,0);
    \path[decorate,decoration={markings,mark=between positions 0.2 and 0.8 step 0.2 with {\draw[dashed](0,0)--+(0,1);}}] (0,0)--(10,0);
  \end{tikzpicture}
  \caption{The Floer strip with $k$ incidence constraints obtained by restricting the curves in \S\ref{sec:prod-oper-with} to the neck; this can be decomposed into a sequence of $k$ sub-strips.}
  \label{fig:substrips}
\end{figure}
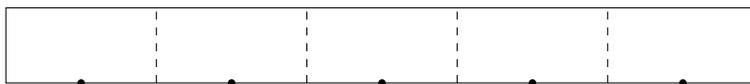

Moreover, since the total energy of a solution $u$ to the moduli space considered in \S\ref{sec:prod-oper-with} is less than $\hbar$ (by construction), the Floer strip in Figure \ref{fig:substrips} also has total energy less than $\hbar$.

By a standard compactness argument (using Arzel\`a-Ascoli), one concludes $k$ infinite-length Floer strips $v_{1},\dots,v_{k}$, each of which has an incidence constraint at $(0,0)$, i.e., $v_{i}(0,0)$ lies in the image of $f_{i}:P_{i}\to L$. By a second limiting process, one can turn off the pertubations so that $\varphi_{1,t}=\phi_{t}\circ R_{\epsilon t}\circ \delta_{1,t}$ converges to $\phi_{t}$; in this fashion, one concludes a sequence $u_{1},\dots,u_{k}$ of infinite-length Floer strips for the original system $\phi_{t}$ so that $u_{i}$ incident to the images of $f_{i}:P_{i}\to L$; see Figure \ref{fig:sequence-floer-strip}.

Note that we do not claim that the positive asymptotic chord $\gamma_{i,+}$ of $u_{i}$ matches the negative asymptotic chord $\gamma_{i+1,-}$ of $u_{i+1}$. However, because the total energy prior to breaking was less than $\hbar$, we conclude that:
\begin{equation*}
  \mathscr{A}_{\phi_{t}}(\gamma_{k,+})\le \mathscr{A}_{\phi_{t}}(\gamma_{i+1,-})\le \mathscr{A}_{\phi_{t}}(\gamma_{i,+})\le \mathscr{A}_{\phi_{t}}(\gamma_{i,-})<\mathscr{A}_{\phi_{t}}(\gamma_{k,+})+\hbar;
\end{equation*}
in particular all of the actions of the asymptotics lie in an interval of length $\hbar$.

We can take $f_{1},\dots,f_{k}$ to have positive codimension and be disjoint from the intersection $L\cap \phi_{1}(L)$ (which is presumed to be an isolated set) while still satisfying the condition that the total intersection product is the point class (as required by Claim \ref{claim:b-sigma}). Then it is clear that each $u_{i}$ is non-constant (otherwise the image of $f_{i}$ would intersect $L\cap \phi_{1}(L)$). Thus we conclude strict inequalities $\mathscr{A}_{\phi_{t}}(\gamma_{i,+})<\mathscr{A}_{\phi_{t}}(\gamma_{i,-})$.

It then follows that: $$\mathscr{A}_{\phi_{t}}(\gamma_{1,-})>\mathscr{A}_{\phi_{t}}(\gamma_{1,+})>\mathscr{A}_{\phi_{t}}(\gamma_{2,+})>\dots>\mathscr{A}_{\phi_{t}}(\gamma_{k,+})$$
are $k+1$ action values contained in an interval of length $\hbar$. Since $\hbar$ can be taken arbitrarily close to $\gamma(\phi_{t})$ in \eqref{eq:defn-hbar}, we conclude $k+1$ action values in an interval of length $\gamma(\phi_{t})$. The argument given in \S\ref{sec:moduli-approach} upgrades this to an interval of length $\gamma(\phi)$, finishing the proof of Theorem \ref{thm:main}.

\subsubsection{A conformal embedding of an infinite strip into a disk}
\label{sec:conf-embedd-strip-disk}
In this section we give a formula for a biholomorphism $w$ between the infinite strip $\R \times [0,1]$ and the closed unit disk $D(1)$ with punctures $\set{-1,1}$.

Consider $w_{1}: \R \times [0,1] \rightarrow \mathbb{H}^{\times}$ defined by $s + it\mapsto e^{\pi(s+it)}$, where $\mathbb{H}^{\times}$ is the punctured closed upper half plane. Note that $w_{1}$ maps the vertical line $\{s\} \times [0,1]$ onto the half circle $C_{s}\coloneq \{e^{\pi(s+it)}: t\in [0,1]\}$. Consider also the M\"obius transformation:
$$w_{2}: \mathbb{H} \rightarrow D(1), \ \ z\mapsto \frac{z-i}{z + i}.$$
This is a biholomorphism between $\mathbb{H}^{\times}$ and $D(1)\setminus \set{-1,1}$ (note that $0$ is sent to $-1$ and $\infty$ is sent to $+1$). The image $w_{2}(C_{s})$ is a circular arc which is orthogonal to the boundary of the disk. The desired biholomorphism is $w=w_{2}\circ w_{1}$; see Figure \ref{fig:habib-1} for an illustration.

\begin{figure}[H]
\centering
\begin{tikzpicture}[scale=0.6]
  \begin{scope}
    \clip (0,0) circle (2);
    \fill[white!80!red] (-1.41421356,-1.41421356) arc (315: 405 : 2) --(-1.41,2)--(1.41,2)-- (1.41421356,1.41421356) arc (135: 225 : 2)--(1.41,-2)--(-1.41,-2)--cycle;
  \end{scope}
  \draw (0,0) circle (2);
  \draw[red] (-1.41421356,-1.41421356) arc (315: 405 : 2);
  \draw[red] (1.41421356,1.41421356) arc (135: 225 : 2);
  \path[every node/.style={inner sep=1pt,fill,circle}] (2,0) node {} (-2,0) node {} (-1.3,0) node{} (1.3,0) node{};
  \path (2,0) node [right] {$w(+\infty)$} (-2,0) node[left] {$w(-\infty)$} (0,0) node {$N(s)$};
\end{tikzpicture}
\caption{The image $N(s)$ of the rectangle $[-s,s]\times [0,1]$ under $w$.}
\label{fig:habib-1}
\end{figure}
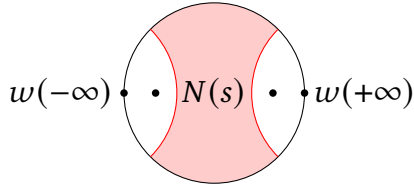

\subsection{Convergence of action values}
\label{sec:cont-spectr-invar}

In this section we prove that the action values $\mathscr{A}_{\phi_{t}}(\mathfrak{e})\in \R$ for classes $\mathfrak{e}\in \mathrm{HF}(\phi_{t})$, as defined by the limiting procedure in \S\ref{sec:HFC-intro}, are well-defined and finite. The crux of the argument is to establish the continuity of action values with respect to changing the slope $\epsilon$ of the approximation $\delta_{t}\circ R^{\alpha}_{\epsilon t}\circ \phi_{t}$.

Consider a Hamiltonian system $\phi_{t}$ supported in $\Omega(r_{0})$. Similarly to \S\ref{sec:HFC-intro}, consider perturbations $\delta_{t}\circ R^{\alpha}_{\epsilon t}\circ \phi_{t}$ so that $\delta_{t}$ is supported in $\Omega(r_{0}+1)$, and which are admissible for definining the Floer complex $\mathrm{CF}(\delta_{t}\circ R^{\alpha}_{\epsilon t}\circ \phi_{t},J)$, as in \S\ref{sec:floer-differential}.

Between any two such perturbation data $(\delta_{-},\epsilon_{-})$ and $(\delta_{+},\epsilon_{+})$ with $\epsilon_{-}\ge \epsilon_{+}$, we take the continuation map associated to the linear interpolation: $$H_{s,t}=(1-\beta(s))K^{-}_{t}+\beta(s)K^{+}_{t},$$ where $\beta: \R \rightarrow [0,1]$ is a non-decreasing cut-off function satisfying $\beta(s) = 0$, for $s\le 0$ and $\beta(s) = 1$, for $s\ge 1$.

Recall that $\mathrm{HF}(\phi_{t})$ is defined as the inverse limit of the groups $\mathrm{HF}(\delta_{t}\circ R^{\alpha}_{\epsilon t}\circ \phi_{t})$ with respect to the above continuation maps. Write $\mathfrak{e}_{\delta, \epsilon}$ for the image of $\mathfrak{e}$ in $\mathrm{HF}(\delta_{t} \circ R^{\alpha}_{\epsilon t} \circ \phi_{t})$. In this section we prove that the limit:
\begin{equation}\label{eq:limit-limit}
  \mathscr{A}_{\phi_{t}}(\mathfrak{e})=\lim_{\epsilon\to 0}\lim_{\delta\to 0}\mathscr{A}(\mathfrak{e}_{\delta,\epsilon})
\end{equation}
exists, where $\mathscr{A}(e_{\delta,\epsilon})$ is defined in \eqref{eq:min-max-action-formula}.

\subsubsection{Hofer norm estimate}
\label{sec:hofer-norm-estimate}

First we establish an inequality in one direction:

\begin{lemma}
Let $(\delta_{-}, \epsilon_{-})$ and $(\delta_{+}, \epsilon_{+})$ be two perturbation data with $\epsilon_{-}\ge \epsilon_{+}$, both smaller than the minimal Reeb period. Let $K^{-}_{t},K^{+}_{t}$ be the normalized Hamiltonian functions generating the systems $\delta_{\pm, t} \circ R^{\alpha}_{\epsilon_{\pm} t}\circ \phi_{t}$, and introduce the error term:
\begin{equation*}
  e(K^{+}_{t},K_{t}^{-})=\int_{0}^{1}\max_{x}(K^{+}_{t}(x) - K^{-}_{t}(x),0)\d t<\infty;
\end{equation*}
the integrand is the maximum of the non-negative part of $K^{+}_{t}(x)-K_{t}^{-}(x)$. Then, for any element $\mathfrak{e}\in \mathrm{HF}(\phi_{t})$, we have the following estimate for the action values:
$$\mathscr{A}(\mathfrak{e}_{\delta_{-}, \epsilon_{-}}) \ge   \mathscr{A}(\mathfrak{e}_{\delta_{+},\epsilon_{+}})- e(K_{t}^{+}, K_{t}^{-}).$$ Moreover, if $\epsilon_{+} = \epsilon_{-}$ then $$|\mathscr{A}(\mathfrak{e}_{\delta_{-}, \epsilon_{-}}) - \mathscr{A}(\mathfrak{e}_{\delta_{+},\epsilon_{+}})| \leq e(K_{t}^{+}, K_{t}^{-}).$$
\end{lemma}
\begin{proof}
  Fix some small constant $\rho>0$. One picks a chain level representative $\sum x_{i}$ for $\mathfrak{e}_{\delta_{+},\epsilon_{+}}$ so that $\min \mathscr{A}(x_{i})\ge \mathscr{A}(\mathfrak{e}_{\delta_{+},\epsilon_{+}})-\rho$. The actions of the chain level sum output by the continuation map $\mathfrak{c}(\sum x_{i})=\sum y_{j}$ can be estimated to yield:
  \begin{equation*}
    \mathscr{A}(\mathfrak{e}_{\delta_{-},\epsilon_{-}})\ge \min\mathscr{A}(y_{j})\ge \min \mathscr{A}(x_{i})-e(K_{t}^{+},K_{t}^{-}).
  \end{equation*}
  The proof of the rightmost inequality is standard and follows from the energy estimate; see, e.g., \cite{hofer-salamon-95,oh-2005-duke,ginzburg-2005-weinstein} or \S\ref{sec:energy-ident-floers}. Taking $\rho\to 0$ yields the desired result.

  To see why $e(K^{+}_{t},K^{-}_{t})$ is finite, observe that $K_{t}^{\pm}=\epsilon_{\pm}r+c_{\pm}$ where $c_{\pm}$ is locally constant outside of $\Omega(r_{0}+1)$, and hence $K^{+}_{t}-K^{-}_{t}=(c_{+}-c_{-})-(\epsilon_{-}-\epsilon_{+})t$; this is bounded from above because we assume $\epsilon_{-}\ge \epsilon_{+}$.
\end{proof}

\subsubsection{Estimating the error term}
\label{sec:estim-error-term}
Let $\phi_{t}$ be a Hamiltonian system generated by a normalized Hamiltonian $H_{t}$. Let $(\delta, \epsilon)$ be an admissible perturbation data; the system $\delta_{t} \circ R^{\alpha}_{\epsilon, t} \circ \phi_{t}$ is generated by:
\begin{equation}\label{eq:formula-for-generating}
  K_{\delta, \epsilon, t} = \Delta_{t}+\epsilon f(r\circ \delta_{t}^{-1}-r_{0})+H_{t}\circ (\delta_{t}R^{\alpha}_{\epsilon t})^{-1}+\epsilon r_{0},
\end{equation}
where $\Delta_{t}$ is the compactly supported Hamiltonian function generating $\delta_{t}$ and the constant term $\epsilon r_{0}$ is used to make $K_{\delta, t}$ normalized. It is convenient to compare with the reference Hamiltonian:
\begin{equation*}
  G_{\epsilon, t}=\epsilon f(r-r_{0})+H_{t}+\epsilon r_{0}.
\end{equation*}
Then the following estimate holds:
\begin{equation}
\label{eq:estimate_pert_data_ham}
|K_{\delta, \epsilon, t} - G_{\epsilon,t}| \leq |\Delta_{t}| + |H|_{C^{1}} |\delta_{t}\circ R^{\alpha}_{\epsilon t}|_{C^{0}} + |f|_{C^{1}} |\delta_{t}|_{C^{0}},
\end{equation}
where $|f|\coloneq \max_{x} |f(x)|$ for a map $f: W \rightarrow \mathbb{R}$ and the $C^{0}$-norm is with respect to a metric that is translation invariant at infinity. Fixing $\rho,\epsilon > 0$, define:
\begin{equation*}
B(\epsilon, \rho)\coloneq  \{ \delta : |K_{\delta, \epsilon} - G_{\epsilon}| < \rho\text{ and }\delta\circ R^{\alpha}_{\epsilon t}\circ \phi_{t}\text{ is admissible for $J$} \}.
\end{equation*}
Similarly let $C(\epsilon_{0},\rho)$ be the set of pairs $(\epsilon,\delta)$ so that $\epsilon<\epsilon_{0}$ and $\delta\in B(\epsilon,\rho)$.

From the estimate \eqref{eq:estimate_pert_data_ham}, one sees that to define Floer cohomology as in \S\ref{sec:HFC-intro} it is enough to take the limit of $\mathrm{HF}(\delta_{t}\circ R^{\alpha}_{\epsilon t}\circ \phi_{t})$ over the subcategory where $(\epsilon,\delta)\in C(\epsilon_{0}, \rho)$.

Associated to the set of perturbations $B(\epsilon,\rho)$, let:
\begin{equation*}
  A(\mathfrak{e};\epsilon,\rho)=\{\mathscr{A}(\mathfrak{e}_{\delta,\epsilon}):\delta\in B(\epsilon,\rho)\};
\end{equation*}
by the estimate \S\ref{sec:hofer-norm-estimate} the diameter of $A(\mathfrak{e};\epsilon,\rho)$ is less than $2\rho$. Associated to the larger set $C(\epsilon_{0},\rho)$, let:
\begin{equation*}
  \bar{A}(\mathfrak{e};\epsilon_{0},\rho)=\set{\mathscr{A}(\mathfrak{e}_{\delta,\epsilon}):(\epsilon,\delta)\in C(\epsilon_{0},\rho)}=\textstyle\bigcup_{\epsilon<\epsilon_{0}}A(\mathfrak{e};\epsilon,\rho).
\end{equation*}
The key to convergence of the action values is to bound the diameter of $\bar{A}(\mathfrak{e};\epsilon_{0},\rho)$.

\subsubsection{Comparing perturbations with different slopes}
\label{sec:comparing-perturbations}

We consider specific pairs: $$\varphi_{\pm,t}=\delta_{\pm,t}\circ R_{\epsilon_{\pm}t}^{\alpha}\circ \phi_{t}.$$ The construction is delicate, as we want to have precise control over the continuation map $\mathrm{HF}(\varphi_{+,t})\to \mathrm{HF}(\varphi_{-,t})$. Fix $\rho,\epsilon_{0}>0$ so that $\epsilon_{0}(1+r_{0})<\rho/3$.

Assume that $\epsilon_{+} < \epsilon_{-}<\epsilon_{0}$ and pick a perturbation data $\delta_{+} \in B(\epsilon_{+}, \rho')$ where $\rho'<\rho/3$. Consider the system $\varphi_{+,t} = \delta_{+,t}\circ R_{\epsilon_{+}t}^{\alpha}\circ \phi_{t}$ which is generated by:
\begin{equation*}
  K^{+}_{t}=\Delta_{+,t}+\epsilon_{+}f(r\circ \delta_{t}^{-1}-r_{0})+H_{t}\circ (\delta_{+,t}R^{\alpha}_{\epsilon_{+}t})^{-1}+\epsilon_{+}r_{0},
\end{equation*}
as in \eqref{eq:formula-for-generating}. Assume that $\Delta_{+,t}$ is supported in $\Omega(r_{0}+0.5)$; this is sufficient to make $\varphi_{+,t}$ admissible. It follows that:
\begin{equation*}
  K^{+}_{t}=\left\{
    \begin{aligned}
      &K^{+}_{t}&&r< r_0+0.5,\\
      &h_{+}(r)&&r_0+0.5\le r< r_{0}+1,\\
      &\epsilon_{+}r+H_{t}&&r_0+1\le r.
    \end{aligned}
  \right.
\end{equation*}
It is important to bear in mind that $H_{t}$ is locally constant on the region $r>r_{0}$. Note that $K^{+}_{t}$ is normalized as $H_{t}$ is normalized.

Define:
\begin{equation*}
  K^{-}_{t}=\left\{
    \begin{aligned}
      &K^{+}_{t}&&r< r_0+0.5,\\
      &h_{-}(r)&&r_0+0.5\le r< r_{0}+1,\\
      &\epsilon_{-} r+H_{t}&&r_0+1\le r,
    \end{aligned}\right.
\end{equation*}
where $h_{-}(r)=h_{+}(r)$ for $r$ near $r_{0}+0.5$, and $\bd_{r}h_{-}(r)\ge \bd_{r}h_{+}(r)$; such a function exists since $\epsilon_{-} > \epsilon_{+}$. We also suppose that $\bd_{r}h_{-},\bd_{r}h_{+}$ are non-negative and everywhere smaller than the minimal period of a Reeb orbit for $\alpha$; this ensures that there are no orbits of $K^{\pm}_{t}$ in the region $r_{0}+0.5\le r$.

We claim that if $\delta_{-,t} \circ R^{\alpha}_{\epsilon_{-}, t} \circ \phi_{t}$ is the system generated by $K_{t}^{-}$, then $\delta_{-} \in B(\epsilon_{-}, \rho)$. To see this we proceed as follows:
\begin{equation}
\begin{aligned}
  |K^{-} - G_{\epsilon_{-}}|
  &= |K^{-} - G_{\epsilon_{-}}|_{\Omega(r_{0} + 1)} \\
  &\leq |K^{-} - K^{+}|_{\Omega(r_{0} + 1)} + |K^{+} - G_{\epsilon_{+}}|_{\Omega(r_{0} + 1)} + |G_{\epsilon_{+}} - G_{\epsilon_{-}}|_{\Omega(r_{0}+1)}\\
  &\leq |h_{-} - h_{+}|_{\Omega(r_{0} + 1)} +  \rho' + (\epsilon_{-} - \epsilon_{+}) (|f(r-r_{0})|_{\Omega(r_{0}+1)} + r_{0})\\
  &\leq \epsilon_{0}(r_{0}+1) +  \rho/3 + \epsilon_{0} (1 + r_{0})<\rho\\
\end{aligned}
\end{equation}

Next, observe that the periodic orbits of the two systems $\varphi_{+,t}$ and $\varphi_{-,t}$ are identical and they all lie in $\Omega(r_{0})$. By a strong maximum principle argument similar to \cite[\S 2]{frauenfelder_schlenk}, all of the Floer continuation cylinders associated to the linear interpolation between $K^{\pm}_{t}$ remain in $\Omega(r_{0})$; see also \cite[\S1.3]{viterbo_functors_and_computations_1}, \cite[\S{D.3}]{ritter_tqft}.

Moreover, since the two Hamiltonians $K^{+}_{t}$ and $K^{-}_{t}$ coincide on $\Omega(r_{0})$, the continuation cylinders solve the translation invariant Floer equation in $\Omega(r_{0})$. By index reasons, the only rigid continuation cylinders are the stationary solutions (i.e., the continuation cylinders are $s$-independent). Moreover, every such cylinder contributes to the continuation map $\mathfrak{c}$, and hence $\mathfrak{c}$ is the ``identity map,'' bearing in mind that the (capped) orbits of $K^{+}_{t}$ and $K^{-}_{t}$ are the same.

This observation implies that the continuation map from $K^{+}_{t}$ to $K^{-}_{t}$ preserves action and is an isomorphism on chain level. Thus $\mathscr{A}(\mathfrak{e}_{\delta_{+}, \epsilon_{+}}) = \mathscr{A}(\mathfrak{e}_{\delta_{-}, \epsilon_{-}})$, as desired.

In particular, this implies that the sets $A(\mathfrak{e};\epsilon_{-},\rho)$ and $A(\mathfrak{e};\epsilon_{+},\rho)$ intersect. Since both sets have diameter bounded by $2\rho$, we conclude:
\begin{prop}
  If $\epsilon_{0}(1+r_{0})<\rho/3$, then the diameter of $\bar{A}(\mathfrak{e};\epsilon_{0},\rho)$ is less than $4\rho$.
\end{prop}
Indeed, the above argument could be performed for any choice of $\epsilon_{+}<\epsilon_{-}<\epsilon_{0}$. This implies that the action values $\mathscr{A}(\mathfrak{e}_{\delta, \epsilon})$ converge in the sense of \eqref{eq:limit-limit}. Standard arguments show that the limit is independent of the choice of contact form, the choice of $r_{0}$, and the choice of complex structure; the details are left to the reader.

\subsection{A priori estimates, transversality, and compactness}
\label{sec:apriori_estimates}

Let $(W,\omega)$ be a convex-at-infinity symplectic manifold, $L \subset W$ a closed Lagrangian submanifold, and $\Sigma$ a compact Riemann surface with boundary and a set $\Gamma$ of interior punctures, and consider an $\omega$-tame almost complex structure $J$ which is translation invariant in the convex end.

To each Hamiltonian connection $\mathfrak{H}$ on $\Sigma\times W$ one can associate the moduli space $\mathscr{M}(\mathfrak{H},J)$ of maps $u:\Sigma\to W$ solving Floer's equation, as explained in \S\ref{sec:floers-equat-hamilt}. In this section we explain the Floer type operations we work with in our paper, in particular we will show that solutions to Floer's equation, for the operations under consideration, satisfy a priori energy bound. Following \cite{brocic_cant} and \cite{hofer-salamon-95} we will shortly review the standard bubbling analysis, transversality and compactness of Floer type moduli spaces. In \S\ref{sec:maximum_principle}, following \cite{brocic_cant}, we show that a soft maximum principle should hold whenever there is a gradient bound and energy bound.

\subsubsection{Cylindrical ends}
\label{sec:cylindrical-ends-1}

Throughout we consider cylindrical coordinates near the punctures, i.e., for each $z\in \Gamma$ one chooses a biholomorphism $\epsilon_{z}: \Sigma_{\pm} \to U_{z} \backslash \{z\} \subset \Sigma$ where $U_{z}$ is a neighborhood of $z$ and $\Sigma_{\pm}=\R_{\pm}\times \R/\Z$.

We assume $\mathfrak{H}$ equals the connection $\mathfrak{H}(\varphi_{t}^{z})$ given by the normalized connection one-form $\mathfrak{a}_{z}=H_{t}^{z}\d t$, where $H_{t}^{z}$ is the normalized generator of the system $\varphi_{t}^{z}$, in the cylindrical end corresponding to $z\in \Gamma$. The results in \S\ref{sec:cylindrical-ends} imply that any connection which is flat in the cylindrical ends can be expressed in this form, after applying some domain-dependent Hamiltonian diffeomorphism near the punctures.

See \cite[\S8d]{seidel_book} for related discussion.

\subsubsection{A priori energy estimates}
\label{sec:energy_estimates}
This section concerns a priori energy bounds for solutions $u\in \mathscr{M}(\mathfrak{H},J)$ if the asymptotics are fixed and the curvature of the corresponding Hamiltonian connection is suitably non-positive at infinity. Suppose first that $\mathfrak{H}$ is a flat Hamiltonian connection on $\Sigma \times W \rightarrow \Sigma$.

If $u\in \mathscr{M}(\mathfrak{H},J)$ is a solution to Floer's equation with contractible asymptotic orbits and contractible boundary loops $u(\bd\Sigma)$ then \S\ref{sec:floers-equat-hamilt} implies the following energy identity:
\begin{equation}\label{eq:not-cutoff}
  E(u)=\int_{\Sigma} u^{*}\omega+\int_{\bd\Sigma}H^{\bd\Sigma}_{t}\d t-\sum_{\Gamma_{+}}\int_{\R/\Z}H^{z}_{t}\d t+\sum_{\Gamma_{-}}\int_{\R/\Z}H_{t}^{z}\d t.
\end{equation}
Note that the contribution due to $\bd\Sigma$ is independent of $u$, and depends only on the constant value $H_{t}^{\bd\Sigma}$ takes on $L$.

Given our set-up, one can consider a cut-off version $\mathfrak{H}^{\sigma}$ where $\sigma\in [0,\infty)$ is a cut-off parameter, defined as follows: for each positive puncture, replace the connection in the cylindrical end by the connection generated by $\beta(s-\sigma)H_{t}^{z}\d t$, where $\beta(x)=1$ for $x\le 0$ and $\beta(x)=0$ for $x\ge 1$. The resulting connection $\mathfrak{H}^{\sigma}$ is no longer flat. However, standard computation of the energy of $u\in \mathscr{M}(\mathfrak{H}^{\sigma},J)$ gives the identity:
\begin{equation}\label{eq:cutoff}
  E(u)=\int_{\Sigma}u^{*}\omega+\int_{\bd\Sigma}H_{t}^{\bd\Sigma}\d t+\sum_{\Gamma_{-}}\int_{\R/\Z}H_{t}^{z}\d t+\sum_{z\in \Gamma_{+}}\int \beta'(s-\sigma)H_{t}^{z}\d s\d t.
\end{equation}
Since $\beta'\le 0$ and $H_{t}^{z}$ is assumed to be positive at infinity, we conclude that $E(u)$ is uniformly bounded in terms of the action of the asymptotics at the negative ends.

One can also deform $\mathfrak{H}$ or $\mathfrak{H}^{\sigma}$ on a compact subset of $\Sigma\times W$. Indeed, $\mathfrak{H}$ is locally given as some one-form $\mathfrak{a}=K\d s+H\d t$ on $\Sigma\times W$, where $K,H$ are normalized. One perturbs $K'=K+\delta_{1},H'=H+\delta_{2}$ where $\delta_{i}$ are compactly supported, defining a perturbed connection $\mathfrak{H}'$. If $u\in \mathscr{M}(\mathfrak{H}',J)$, then the results in \S\ref{sec:energy-ident-floers} give an energy identity of the form:
\begin{equation*}
  E(u)=\int_{\Sigma} u^{*}\omega+\int_{\bd\Sigma}H^{\bd\Sigma}_{t}\d t-\sum_{\Gamma_{+}}\int_{\R/\Z}H^{z}_{t}\d t+\sum_{\Gamma_{-}}\int_{\R/\Z}H_{t}^{z}\d t+\int_{\Sigma}u^{*}\mathfrak{r}.
\end{equation*}
The new term $\mathfrak{r}$ is the curvature two-form on $\Sigma\times W$ for $\mathfrak{H}'$ and is compactly supported and vanishes on vectors tangent to $W$. The additional contribution to the energy can be uniformly bounded in terms of the $C^{1}$ size of the perturbation $\delta$; see \S\ref{sec:curv-hamilt-conn} for the precise formula for the curvature term.

One similarly obtains an energy identity for compactly supported perturbations $\mathfrak{H}^{\sigma,\prime}$ of the cut-off connection by adding the $\int u^{*}\mathfrak{r}$ term to \eqref{eq:cutoff}.

\subsubsection{Bubbling}
\label{sec:bubbling}

The bubbling analysis required is standard and our arguments follow \cite[\S A]{hofer-salamon-95}, \cite[\S4]{mcduffsalamon}, and \cite[\S C]{cant_chen}; see also \S\ref{sec:tame-sympl-manif} for consideration of the non-compact end of $W$.

One shows that an a priori energy bound for a sequence of solutions to Floer's equation guarantees a gradient bound unless bubbling occurs, see, e.g., \cite[\S A]{hofer-salamon-95}. The argument, with some necessary adjustments, works in the general setting of Floer's equation associated to a Hamiltonian connection with Lagrangian boundary conditions, as in \S\ref{sec:floers-equat-hamilt} and \cite[\S8]{mcduffsalamon}.

Because we are using the semipositivity framework to achieve gradient bounds, it is important to know that bubbling implies the existence of a holomorphic sphere incident to the limiting solution to Floer's equation; see \cite[Theorem A.1.(iii)]{hofer-salamon-95}.

\subsubsection{Gradient bounds and elliptic regularity}
\label{sec:gradient_bounds}
In this paper, we work with a perturbed Cauchy-Riemann equation with a domain dependent smooth perturbation term $A_{z}$,
\begin{equation}\label{eq:H_LBC}
\begin{aligned}
\partial_{s}u + J_{z}(u)\partial_{t}u & = A_{z}(u),
\end{aligned}
 \end{equation}
and we are interested in solutions with Lagrangian boundary conditions $u(\partial \Sigma(1)) \subset L$ where $\Sigma(1) = D(1) \text{ or } D(1) \cap \overline{\mathbb{H}}$ and $J$ is a domain dependent almost complex structure. Indeed, solutions $u$ to the equation above with gradient bounded by a constant $c_{1}$, satisfy $|u|_{W^{k,2}} \leq c_{k}$ on compact subsets for all $k > 0$ where $c_{k}$ only depends on $c_{1}, J, L, A$. This is implicitly used in compactness results and bubbling analysis, see \S\ref{sec:bubbling}. We refer to \cite[\S C]{robbinsalamon} and \cite{brocic_cant} for further details.

\subsubsection{Maximum principle}
\label{sec:maximum_principle}
In this subsection assume that the elements of the moduli space $\mathscr{M}(\mathfrak{H},J)$ satisfy the gradient bound and the energy bound condition, i.e., for any sequence of points $\{z_{n}\}_{n\geq 1}$ on the domain $\Sigma$ and any sequence $\{u_{n}\}_{n\geq 1}$ of solutions the two sequences $|\nabla u_{n}(z_{n})|, E(u_{n})$ are bounded, where the norm is with respect to some translation invariant metric at infinity on $W$, and $E(u)$ is the energy of $u$.

Define $\Sigma_{0}\subset \Sigma$ to be a compact complement of cylindrical ends, so that $\mathfrak{H},J$ are translation invariant outside of $\Sigma_{0}$.

We first claim that there exists $\sigma > 0$ such that $u_{n}(\Sigma\setminus \Sigma_{0}) \subset \Omega(e^{\sigma}r_{0})$ for all $n$, where $\Omega$ is a star-shaped domain in $W$. If it does not hold, then for every $\sigma > 0$ there is $n$ so that $u_{n}(\Sigma\setminus \Sigma_{0}) \cap \rho_{\sigma}(\Omega) \neq \emptyset$. Then \cite[Proposition 2.2]{brocic_cant} implies that the sequence $u_{n}$ can not have bounded energy, which is a contradiction with our assumption.

Therefore, if the maximum principle fails for the sequence $\{u_{n}\}_{n\geq 1}$, it must fail on the compact subset $\Sigma_{0}$. Since, the image of $\Sigma_{0}$ under $u_{n}$ intersects a fixed compact subset of $W$ for every $n$ and the sequence $\{u_{n}\}_{n\geq 1}$ satisfies the gradient bound, so the images $u_{n}(\Sigma_{0})$ can not go off to infinity. Hence, the sequence $\{u_{n}\}_{n\geq 1}$ satisfies the maximum principle. See \cite{groman-maximum-principle} for related discussion.

\subsubsection{Asymptotics at punctures}
\label{sec:asymptotics_at_punctures}
Our arguments rely on standard asymptotic convergence results for solutions to the translation invariant Floer's equation in cylindrical or strip-like ends; we refer the reader to \cite[\S 1.5]{salamon1997} and \cite[\S 4]{robbinsalamon}.

\subsubsection{Transversality}
\label{sec:transversality}
Our approach to achieving transversality for moduli spaces of solutions to Floer's equation follows \cite[\S8]{mcduffsalamon}; see also \cite[\S4.2]{schwarz-thesis}, \cite{hofer-salamon-95,floer_hofer_salamon_transversality,wendl-sft}, and \cite[\S4.1]{brocic_cant}.

The strategy to achieve transversality is to perturb the Hamiltonian connection $\mathfrak{H}$ in a compactly supported fashion. More precisely, one considers Hamiltonian connections $\mathfrak{H}+\delta$ which agree with $\mathfrak{H}$ outside of a compact subset of $W\times \Sigma$. Typically we have that $\mathfrak{H}$ is a flat connection, but we do not require that the perturbations $\mathfrak{H}+\delta$ are flat; however, since the perturbations are compactly supported, we will still have the required a priori estimates, as explained in \S\ref{sec:energy_estimates}.

Fix a precompact open set $U$ in $W\times \Sigma$ large enough that every solution to Floer's equation passes through $U$; e.g., one can take $U$ to be $\Omega(r)\times D$ where $D$ is a disk in $\Sigma$ and $r$ is a sufficiently large number. Following the standard strategy, one considers $\mathscr{P}$ to be a sufficiently rich Banach space of perturbation data $\delta$ compactly supported in $U$. One then considers the universal moduli space $\mathscr{M}_{\mathrm{uni}}(\mathfrak{H})$ of solutions $(u,\delta)$ where $u$ solves the equation for the perturbed system $\mathfrak{H}+\delta$. Similarly to the proof of \cite[Theorem 8.3.1]{mcduffsalamon}, one proves that $\mathscr{M}_{\mathrm{uni}}(\mathfrak{H})$ is cut transversally, and that any regular value $\delta$ of the projection $\mathscr{M}_{\mathrm{uni}}(\mathfrak{H})\to \mathscr{P}$ will make $\mathscr{M}(\mathfrak{H}+\delta)$ cut transversally.

The construction can also be done parametrically, i.e., given a one-parameter family of data $\mathfrak{H}_{\tau},j_{\tau},J_{\tau}$, for $\tau\in [0,1]$; this parametric transversality is necessary for the deformation arguments employed in \S\ref{sec:pair-pants-product}, \S\ref{sec:deforming-pair-pants}; see also \S\ref{sec:semip-param-moduli}.

\subsection{Semipositivity}
\label{sec:semipositivity}

The idea behind the semipositive (i.e., weakly monotone) condition is to preclude the bubbling phenomenon in \S\ref{sec:bubbling} by controlling the dimension of moduli spaces of holomorphic spheres. This is achieved by constraining which Chern numbers can appear.

\begin{defn}
A symplectic manifold $(W,\omega)$ is called \emph{semipositive} if, for every $A\in \pi_{2}(W)$, the following holds:
\begin{equation*}
  \omega(A) > 0\text{ and }c_{1}(A) \geq 3 - n \implies c_{1}(A)\geq 0.
\end{equation*}
\end{defn}

Floer theory was constructed in \cite{hofer-salamon-95} for semipositive symplectic manifolds (weakly monotone in their terminology). The Arnol'd conjecture (with the $\F_2$-Betti numbers) was proved for all semipositive symplectic manifolds in \cite{ono95}.

\subsubsection{Semipositivity and Hamiltonian Floer theory}
\label{sec:semip-hamilt-floer}
Consider a moduli space $\mathscr{M}(\mathfrak{H},J)$ of solutions to Floer's equation, where $\mathfrak{H}$ is a Hamiltonian connection on $\Sigma \times W$ and $\Sigma$ is a punctured Riemann surface; see \S\ref{sec:floers-equat-hamilt}.

If $J$ is chosen generically and $\mathfrak{H}'$ is a generic perturbation of the Hamiltonian connection on the compact part of $W$, one can guarantee that every sequence $u_{n}\in \mathscr{M}(\mathfrak{H}',J)$ with index $\le 1$ and bounded energy has bounded first derivatives with respect to a metric which is translation invariant in the ends.

The genericity is used to ensure that the evaluation map of the moduli space of simple $J$-holomorphic spheres \eqref{eq:semipositive-2} is transverse to the evaluation map $\mathscr{M}(\mathfrak{H}',J)\to W$ of the Floer moduli space. With the above bound on the index of the solutions $u_{n}$, we can preclude the formation of bubbles.

To see this one argues as follows: semipositivity implies that bubbles can only have Chern number either $0$ or $1$. Holomorphic spheres with zero first Chern number form a codimension $4$ pseudocycle in $W$ and hence they generically miss the $3$-dimensional pseudocycle defined by the moduli space of solutions with index $\leq 1$. The bubbles with Chern number $1$ can only touch solutions with index $0$, i.e., $1$-periodic orbits of the Hamiltonian system. The latter also does not happen generically since the moduli space of holomorphic spheres with Chern number $1$ form a codimension $2$ pseudocycle. Finally the claim follows from bubbling analysis in \S\ref{sec:bubbling}. See \cite[\S3]{hofer-salamon-95} and \cite[\S3]{mcduffsalamon} for more details.

\subsubsection{Semipositivity and parametric moduli spaces in Hamiltonian Floer theory}
\label{sec:semip-param-moduli}
The argument in \S\ref{sec:semip-hamilt-floer} can be done parametrically, as follows.

Fix a one-parameter family $\mathfrak{H}_{\theta},j_{\theta}$ for $\theta\in [0,1]$, where $j_{\theta}$ is a family of conformal structures on $\Sigma$ (i.e., unlike the preceding discussion, the Riemann surface structure is not fixed).

One considers the parametric moduli space $\mathscr{M}(\mathfrak{H}_{\theta},J,j_{\theta})$ whose solutions are pairs $(\theta,u)$ where $u$ solves Floer's equation for $\mathscr{M}(\mathfrak{H}_{\theta},J)$ on the Riemann surface $(\Sigma,j_{\theta})$. As above, there is a natural evaluation map:
\begin{equation}\label{eq:ev_param}
  (\theta,u,z)\in \mathscr{M}(\mathfrak{H}_{\theta},J,j_{\theta})\times \Sigma\mapsto u(z)\in W,
\end{equation}
which defines a pseudochain of dimension $2+\dim \mathscr{M}(\mathfrak{H}_{\theta},J,j_{\theta})$. If all components of $\mathscr{M}$ are cut transversally, then a sphere of Chern number $\ge 1$ cannot bubble off along any sequence whose index in $\mathscr{M}$ is $0$ or $1$, while a sphere of Chern number $0$ cannot bubble off such a sequence because the evaluation map on the moduli space of simple holomorphic spheres is codimension at least $4$ in this case.

Thus one can use bubbling analysis to derive a priori $C^{1}$ estimates on $0$ or $1$ dimensional components of parametric moduli spaces. This is used, e.g., to prove unit times unit is unit in \S\ref{sec:unit-times-unit}.

\subsection{Tame symplectic manifolds}
\label{sec:tame-sympl-manif}
A symplectic manifold $(W,\omega,J)$ with almost complex structure $J$ is called \emph{tame} if there is a metric $g$ so that $\omega(v,Jv)/g(v,v)$ is uniformly positive for $v\ne 0$, and $(W,J,g)$ is tame as an almost complex manifold as in \S\ref{sec:tame-almost-complex-manifold}. See \cite[2.3.{$\mathrm{A}'$}]{gromov85}, \cite{polterovich_lag_displacement_energy,chekanov_1998}. Note that, in the case where $W$ is compact, tameness is equivalent to $\omega(v,Jv)>0$ for $v\ne 0$.

\subsubsection{Tame almost complex manifolds}
\label{sec:tame-almost-complex-manifold}
A manifold $(W,J,g)$ with a complete metric $g$ and almost complex structure $J$ is called \emph{tame} if there are constants $\delta>0$ and $C>0$ so that:
\begin{enumerate}
\item $J$ acts by $g$-isometries in each tangent space,
\item $W$ can be covered by coordinate charts identified with $B(\delta)\subset \mathbb{C}^{n}$ so that the $C^{2}$ sizes of $g$ and $J$ are bounded by $C$ in each chart.
\end{enumerate}
It follows that the injectivity radius of $g$ is bounded from below, and the sectional curvatures are uniformly bounded; see \cite[2.3.{$\mathrm{A}'$}]{gromov85}.

\subsubsection{Mean value property}
\label{sec:mean_value_property}
Let $(M^{2n},J,g)$ be a tame almost complex manifold. The mean value property for the energy density states that there are positive constants $c$ and $\epsilon$ depending on $J,g$, such that the following holds:
\begin{equation*}
  \int_{D(r)}|\d u|_{g}^{2}\d s\d t < \epsilon \implies |\d u(0)|_{g}^{2} \leq \frac{c}{r^{2}}\int_{D(r)}|\d u|_{g}^{2}\d s\d t,
\end{equation*}
for all $J$-holomorphic curves $u: D(r)\rightarrow M$, where $D(r)$ is the standard disk with radius $r$; see \cite{robbinsalamon,mcduffsalamon,cant_chen}. A similar result holds for $J$-holomorphic half-disks whose real-part lies on a compact totally real submanifold.

\subsubsection{Diameters of low energy annuli in tame almost complex manifolds}
\label{sec:diff-ineq-annuli}
This section concerns removal of singularities, and is used in the bubbling analysis in \S\ref{sec:bubbling}. Let $A(r)$ be the annulus (or rectangle) $[-r,r]\times S$, with $S=\R/\Z$ or $S=[0,1]$.

\begin{prop}
  Let $(W,J)$ be tame and $L$ be a compact totally real submanifold. For every $\delta>0$ there exists $\epsilon>0$ so that if $u:A(r+1)\to W$ is $J$-holomorphic then:
  \begin{equation*}
    \int_{A(r+1)} \abs{\bd_{s}u}^{2}_{g}\d s\d t\le \epsilon\implies \mathrm{diam}_{g}(u(A(r)))\le \delta.
  \end{equation*}
  Note that $\epsilon$ does not depend on $r$. In the case when $S=[0,1]$, we require $u$ to map the $t=0,1$ boundaries into $L$.
\end{prop}
\begin{proof}
  One uses the mean-value property from \S\ref{sec:mean_value_property} and proves the energy of $u$ over smaller domains $A(s)$ decays exponentially as $s$ approaches $0$. For the detailed argument, see \cite[\S6]{cant_chen} and \cite[\S4]{mcduffsalamon}.
\end{proof}

Suppose that $u:\mathbb{C}\to W$ or $u:\mathbb{H}\to (W,L)$ is holomorphic with finite energy. The non-compact end of $\mathbb{C}$ or $\mathbb{H}$ can be covered by pieces biholomorphic to $A(r)$, and so that $u$ has small energy on these pieces. One thereby concludes the image of $u$ has finite diameter, and $u$ has a continuously removable singularity at $\infty$. With a bit more work, one can show that the extension is of sufficiently regularity ($W^{1,p}$ for $p$ bigger than $2$) in order to conclude the continuous extension of $u$ is a smooth holomorphic sphere or disk with boundary on $L$; this smooth removal of singularities is explained in greater detail in \cite{mcduffsalamon}.

\appendix
\section{Flat Hamiltonian connections}
\label{sec:curvature}

Our approach to Floer's equation on general Riemann surfaces is via bundles with a Hamiltonian connection. See \cite[\S8]{mcduffsalamon}, \cite[\S7]{seidel_thesis}, \cite[\S9.3]{polterovich_book} for a similar approach.

\subsection{Ehresmann connections}
\label{sec:ehresm-conn}
To every smooth fiber bundle $\pi:E\to B$ one associates the \emph{vertical sub-bundle} $V=\ker \d\pi \subset TE$. An \emph{Ehresmann connection} is a smoothly varying choice of linear complement $\mathfrak{H}\subset TE$.

\subsubsection{Complete connections}
\label{sec:complete-connections}
Let $E\to B$ be a fiber bundle with an Ehresmann connection $\mathfrak{H}$. Every vector field $Y$ on $B$ has a unique lift to a horizontal vector field $Y_{\mathfrak{H}}\in \mathfrak{H}$. An Ehresmann connection is called \emph{complete} provided that every compactly supported vector field $Y$ on $B$ lifts to a complete vector field on $E$.

\subsubsection{Monodromy diffeomorphisms}
\label{sec:monodr-diff}

Let $b(t)$ be a path in $B$. If $\mathfrak{H}$ is a complete Ehresmann connection on $E\to B$, then for every $e_{0}\in E_{b(0)}$ there is a unique \emph{horizontal lift} $e(t)$ so $e(0)=e_{0}$, $\pi(e(t))=b(t)$, and $e'(t)\in \mathfrak{H}$. The map which associates $e_{0}$ to $e(1)$ is a diffeomorphism $E_{x(0)}\to E_{x(1)}$ and is called the \emph{monodromy} of $b(t)$.

As an example, solutions of the ODE $y'(x)=F(x,y(x))$, $y(0)=y_{0}$, are horizontal lifts of $x(t)=t$ for the Ehresmann connection $\mathfrak{H}=\set{\d y-F(x,y(x))\d x=0}$ on $\R^{2}$. The map which associates an initial condition $y_{0}$ to $y(1)$ is the prototypical example of a monodromy diffeomorphism.

\subsubsection{Flat connections}
\label{sec:flat-connections}

An Ehresmann connection $\mathfrak{H}$ is called \emph{flat} provided that, for every choice of $e\in E_{b}$, there exists a germ of a section $\mathfrak{s}$ of $E\to B$ at $b$ satisfying $\mathfrak{s}(b)=e$ and $\im(\d \mathfrak{s})=\mathfrak{H}$. If $\mathfrak{H}$ is a complete flat connection, then deformations of a path $b(t)$ relative its endpoints do not change the monodromy.

\subsubsection{Curvature of an Ehresmann connection}
\label{sec:curv-ehresmann}

Let $Y_{1},Y_{2}$ be two vector fields on $B$ defined in a neighborhood of $b$. Let $Y_{i,\mathfrak{H}}$ denote their horizontal lifts to $E$. For $e\in E_{b}$, define:
\begin{equation*}
  R_{\mathfrak{H},e}(Y_{1},Y_{2})\coloneq[Y_{1,\mathfrak{H}},Y_{2,\mathfrak{H}}]-[Y_{1},Y_{2}]_{\mathfrak{H}}.
\end{equation*}
It is a standard exercise in manipulating the Lie bracket to show that $R_{\mathfrak{H},e}(Y_{1},Y_{2})$ is valued in the vertical sub-bundle $V\subset TE$. Moreover, $R_{\mathfrak{H},e}$ commutes with multiplication by smooth functions; in particular, $R_{\mathfrak{H},e}$ is induced by a tensor $TB_{b}\wedge TB_{b}\to V_{e}$. The resulting tensor $R_{\mathfrak{H}}:\pi^{*}(TB\wedge TB)\to V$ is called the \emph{curvature tensor} of $\mathfrak{H}$.

It is clear that connection is flat if and only if its curvature tensor is everywhere zero.

\subsection{Hamiltonian connections on trivial bundles}
\label{sec:hamilt-conn-trivial}

Let $E=W\times B\to B$ be a trivial bundle and suppose the fiber $(W,\omega)$ is a symplectic manifold. A \emph{Hamiltonian connection} is an Ehresmann connection $TE=TW\oplus \mathfrak{H}$ (where $TW$ is identified with the vertical sub-bundle $V$ of the fibration $E\to B$) so that:
\begin{enumerate}
\item $\mathfrak{H}$ is the $\Omega$-complement to $TW$,
\item $\Omega\coloneq\pr^{*}\omega-\d \mathfrak{a}$, and,
\item $\mathfrak{a}$ is a one-form on $E$ which vanishes on $TW$.
\end{enumerate}
It follows that $\mathfrak{H}$ is a linear complement to $TW$ and hence defines an Ehresmann connection.

\subsubsection{One-forms vanishing on the vertical bundle}
\label{sec:one-forms-vanishing-vert}
If $\mathfrak{a}$ is a one-form on $E$ which vanishes on $TW$, and $x_{1},\dots,x_{k}$ are coordinates on $B$ pulled back to coordinates on $E$, then $\mathfrak{a}$ can be written as $\mathfrak{a}=H_{1}\d x_{1}+\dots+H_{k}\d x_{k}$. Here $H_{i}$ is considered a function on $W\times B$, i.e., it is domain dependent. In this sense, $\mathfrak{a}$ can be considered as a one-form on $B$ taking values in $C^{\infty}(W,\R)$; see, e.g., \cite[\S8e]{seidel_book}, \cite[pp.\,3293]{kislev_shelukhin}.

\subsubsection{Monodromy of a Hamiltonian connection is Hamiltonian}
\label{sec:monodr-hamilt-conn}

Fix local coordinates $x_{1},\dots,x_{k}$ on the base $B$ of the trivial bundle $W\times B\to B$. Consider the Hamiltonian connection $\mathfrak{H}$ induced by $\mathfrak{a}=H_{1}\d x_{1}+\dots+H_{k}\d x_{k}$ as in \S\ref{sec:one-forms-vanishing-vert}. Let $x:[0,1]\to B$ be a path in $B$ remaining in the local coordinate chart. Then the induced monodromy diffeomorphism $W\to W$ is given by the time-one map of the Hamiltonian system:
\begin{equation*}
  \gamma'(t)=\textstyle\sum_{i=1}^{k}x_{i}'(t)X_{H_{i}}(\gamma(t)).
\end{equation*}
Indeed, the velocity of $(x(t),\gamma(t))$ is $\Omega$-orthogonal to $V=TW$ when $\Omega=\pr^{*}\omega-\d \mathfrak{a}$, as can be checked directly. As a consequence, the monodromy of any Hamiltonian connection along any path is a Hamiltonian diffeomorphism.

\subsubsection{Curvature of a Hamiltonian connection is Hamiltonian}
\label{sec:curv-hamilt-conn}

Let $\mathfrak{H}$ be the Hamiltonian connection induced by a one-form $\mathfrak{a}$ as above. Then:
\begin{equation}\label{eq:curvature_formula}
  R_{\mathfrak{H}}(\bd_{i},\bd_{j})\intprod \omega=-\d (\pd{H_{i}}{x_{j}}-\pd{H_{j}}{x_{i}}+\set{H_{i},H_{j}})=-\d \mathfrak{r}(\bd_{i},\bd_{j}).
\end{equation}
In particular, the vertical curvature vectors define a Hamiltonian vector field on $W$, and the generating Hamiltonians can be encoded as a curvature two-form $\mathfrak{r}$. See \cite[pp.\,3293]{kislev_shelukhin} for similar formula.

Note that if $\mathfrak{H}$ is flat, then $\mathfrak{r}$ is not necessarily zero, but is the pullback of a two-form from $B$ (assuming $W$ is connected). Typically the way to ensure that $\mathfrak{r}=0$ for flat connections is to enforce some normalization conditions on the Hamiltonian functions appearing in $\mathfrak{a}$; see \S\ref{sec:norm-cond-2}.

\subsubsection{Coordinate changes for Hamiltonian connections}
\label{sec:coordinate-changes-hamilt-conn}
The following lemma is key for constructing and manipulating Hamiltonian connections; it ensures that the class of Hamiltonian connections is closed under a large family of coordinate changes.

\begin{lemma}
  Let $E=W\times B$. A contractible family $g_{x}$ of Hamiltonian diffeomorphisms, where $x\in B$, has a total map $g(w,x)=(g_{x}(w),x)$ which satisfies:
  \begin{equation*}
    g^{*}\pr^{*}\omega=\pr^{*}\omega+\d \mathfrak{b},
  \end{equation*}
  where $\mathfrak{b}$ is a one-form which vanishes on the vertical bundle.
\end{lemma}
\begin{proof}
  Indeed, if $g_{x,t}$ is a path of systems $g_{x,1}=g_{x}$ and $g_{x,0}=\id$, then differentiating $g_{t}(w,x)=(g_{x,t}(w),x)$ with respect to $t$ yields:
  \begin{equation}\label{eq:coordinate_change_formula}
    \pd{}{t}g^{*}_{t}\pr^{*}\omega=\d(g_{t}^{*}\pr^{*}[\d H_{x,t}])=-\d(\sum_{i=1}^{n}\pd{}{x_{i}}(H_{x,t}\circ g_{x,t})\d x_{i})=\d \beta_{t},
  \end{equation}
  where $H_{x,t}$ are the generators of the systems $t\mapsto g_{x,t}$. The existence of the path $g_{x,t}$ is what we mean when we say $g_{x,t}$ is a contractible family.

  Integrating $\beta_{t}$ over $t\in [0,1]$ constructs the primitive one-form $\mathfrak{b}$ which vanishes on the vertical bundle. In particular, if $\mathfrak{H}_{2}$ is a Hamiltonian connection, then $\mathfrak{H}_{1}=g^{-1}_{*}\mathfrak{H}_{2}$ is also Hamiltonian.
\end{proof}

\subsubsection{Contact-at-infinity Hamiltonian connections}
\label{sec:contact-at-infinity}

A Hamiltonian connection $\mathfrak{H}$ is called \emph{contact-at-infinity} if the one-form $\mathfrak{a}$ satisfies that $\mathfrak{a}(X)$ is one-homogeneous with respect to the Liouville flow, up to the addition of a function which is constant outside of a compact set, for all tangent vectors $X\in TB_{x}$. Here $\mathfrak{a}(X)$ is defined by lifting $X$ arbitrarily to a family of vectors along the fiber $W\times \set{x}$.

Similarly, one says that $\mathfrak{H}$ is $\alpha$-Reeb outside of $\Omega(r_{0}+1)$ provided each function $\mathfrak{a}(X)$ equals $r$, up to the addition of a function which is constant outside of $\Omega(r_{0}+1)$, for all tangent vectors $X\in TB_{x}$. Here recall that $\Omega(r_{0}+1)$ is the starshaped domain and $r$ is the radial function associated to the Reeb flow for $\alpha$, as explained in \S\ref{sec:reeb-flow-convex-end}.

If $\mathrm{Ham}(W)$ denotes the group of contact-at-infinity Hamiltonians, then contractible families $g_{x}\in \mathrm{Ham}(W)$ preserve the class of contact-at-infinity Hamiltonians, as can be checked by the formula for $\mathfrak{b}$ in \S\ref{sec:coordinate-changes-hamilt-conn}.

Similarly, the group $\mathrm{Ham}(W;R^{\alpha},r_{0}+1)$ of Hamiltonian diffeomorphisms which agree with an $\alpha$-Reeb flow outside of $\Omega(r_{0}+1)$ acts on the space of connections which are $\alpha$-Reeb outside of $\Omega(r_{0}+1)$.

Henceforth, we fix the parameter $r_{0}$ and say $\mathfrak{H}$ is $\alpha$-Reeb when it is $\alpha$-Reeb outside of $\Omega(r_{0}+1)$. The consideration of this particular class of connections is important for establishing the energy estimate in \S\ref{sec:prod-oper-with}.

The results of \S\ref{sec:monodr-hamilt-conn} specialize and imply that the monodromy of a contact-at-infinity, resp., $\alpha$-Reeb, connection $\mathfrak{H}$ is valued in $\mathrm{Ham}(W)$, resp., $\mathrm{Ham}(W;R^{\alpha},r_{0}+1)$.

\subsubsection{Normalization conditions}
\label{sec:norm-cond-2}

We say that a one-form $\mathfrak{a}$ on $W\times B$ (which vanishes on vertical vectors) is \emph{normalized} if $\mathfrak{a}(X)$ is normalized according to \S\ref{sec:norm-cond} on the fiber $W\times\set{x}$ for every vector $X\in TB_{x}$. Every contact-at-infinity, resp $\alpha$-Reeb, Hamiltonian connection can be generated by such a normalized one-form.

An important property of normalized connections $\mathfrak{H}$ is that the corresponding curvature two-form $\mathfrak{r}$ appearing in \eqref{eq:curvature_formula} is normalized, and if the connection is flat the curvature is identically zero. Moreover, if normalized one-forms $\mathfrak{a}$ and $\mathfrak{b}$ generate the same connection $\mathfrak{H}$ then $\mathfrak{a} = \mathfrak{b}$.

\subsubsection{Conjugation of monodromy and coordinate changes}
\label{sec:conj-monodr-coord}

Let $\mathfrak{H}$ be a flat Hamiltonian connection on $W\times B$. Let $g_{x}$ represent a coordinate change as in \S\ref{sec:coordinate-changes-hamilt-conn}. Given a path $x(t)$ in the base, we have:
\begin{equation*}
  g_{x(1)}^{-1}\circ (\text{monodromy of $g_{*}\mathfrak{H}$ along $x(t)$})\circ g_{x(0)}=(\text{monodromy of $\mathfrak{H}$ along $x(t)$}).
\end{equation*}
This observation is used in \S\ref{sec:prod-oper-with} to obtain a connection with a desired monodromy.

\subsection{Locally trivial flat Hamiltonian connections}
\label{sec:locally-trivial-flat}

In this section, we consider bundles with structure group $\mathrm{Ham}(W)$ or $\mathrm{Ham}(W;R^{\alpha},r_{0}+1)$. The results in this section are phrased using $\mathrm{Ham}(W)$ for notational simplicity, although the constructions in this section work equally well if $\mathrm{Ham}(W)$ is replaced throughout with $\mathrm{Ham}(W;R^{\alpha},r_{0}+1)$.

\subsubsection{Locally trivial flat connection}
\label{sec:locally-trivial-flat-hamiltonian}

Let $E\to \Sigma$ be a locally trivial bundle with fiber $W$ and structure group $\mathrm{Ham}(W)$. An Ehresmann connection $\mathfrak{H}$ on $E$ is called \emph{flat} if around each point of the base there is a chart around the point so that the connection appears flat as in \S\ref{sec:flat-connections}. Such a bundle $(\pi:E\to \Sigma,\mathfrak{H})$ will be called a \emph{locally trivial Hamiltonian bundle with a flat connection}.

For example, if one can pass to a subatlas where the changes of trivialization are locally constant on each intersection $U_{0}\cap U_{1}$, then the locally defined flat connections glue together to form a flat connection $\mathfrak{H}$.

\subsubsection{Pairs of pants}
\label{sec:pair-pants}

The pair-of-pants $\Sigma$ is a two-sphere with three punctures; we can fix this as $\mathrm{CP}^{1}$ with the punctures at $0,1$ and $\infty$.
\begin{prop}\label{prop:pop-covering}
  There is a simply-connected Riemann surface $\Sigma'$ with a free-and-proper action of the free-group $\mathrm{F}_{2}$ by biholomorphisms, and a biholomorphism $\Sigma'/\mathrm{F}_{2}\to \Sigma$. The monodromy of the covering space around $0$ equals the action of the first generator of $\mathrm{F}_{2}$, and the monodromy around the $1$ equals the second generator of $\mathrm{F}_{2}$.
\end{prop}
\begin{proof}
  This is a consequence of the theory of covering spaces in the context of Riemann surfaces; see, e.g., \cite{donaldson}.
\end{proof}

This proposition gives an efficient construction of a smooth family of flat Hamiltonian connections with any desired monodromy around $0$ and around $1$:

\begin{prop}\label{prop:pop-construction}
  Fix any choice of monodromy $\mathrm{F}_{2}\to \mathrm{Ham}(W)$. The diagonal quotient: $$E:(W\times \Sigma')/\mathrm{F}_{2}\to \Sigma'/\mathrm{F}_{2}\simeq \Sigma$$ is a locally trivial Hamiltonian bundle with a flat connection. Moreover, this bundle is trivial as a bundle with Hamiltonian structure group.
\end{prop}
\begin{proof}
  Take contractible open sets $U$ in $\Sigma$ which admit smooth lifts to $\Sigma'$; each choice of lift gives an identification of $E|_{U}$ with $W\times U$. It is straightforward to check that transition functions are locally constant, which completes the first part of the proof (by the definition of locally flat in \S\ref{sec:locally-trivial-flat-hamiltonian}).

  For the second part, note that the map $\mathrm{F}_{2}\to \mathrm{Ham}(W)$ can be continuously homotoped to the constant map through group homomorphisms. This process produces a bundle over $\Sigma\times [0,1]$. The bundle restricted to $\Sigma\times \set{1}$ is isomorphic to the restriction over $\Sigma\times \set{0}$, which is clearly the trivial bundle $\Sigma\times W$.
\end{proof}

\subsubsection{Locally trivial Hamiltonian bundles on the pair-of-pants are trivial}
\label{sec:vcech-cocycles}

Fix a Hamiltonian bundle with flat connection over the pair-of-pants $E\to \Sigma$. The goal in this section is to prove the following result allowing us to realize every locally trivial bundle with flat connection as a trivial bundle with a Hamiltonian connection in the sense of \S\ref{sec:hamilt-conn-trivial}.

\begin{prop}
  Every locally trivial Hamiltonian bundle with flat connection $\mathfrak{H}$ on the pair-of-pants $\Sigma$, as in \S\ref{sec:locally-trivial-flat-hamiltonian}, can be trivialized in such a way that $\mathfrak{H}$ is Hamiltonian with connection form $\Omega=\pr^{*}\omega-\d\mathfrak{a}$, as in \S\ref{sec:hamilt-conn-trivial}. Moreover, we may suppose that $\mathfrak{a}$ is normalized as in \S\ref{sec:norm-cond}.
\end{prop}
\begin{proof}
  Using the covering space furnished by Proposition \ref{prop:pop-covering}, one can show that $E$ is isomorphic to one of the bundles constructed in Proposition \ref{prop:pop-construction}, and hence is trivial, in the category of fiber bundles with structure group $\mathrm{Ham}(W)$.

  Standard ideas in Cech-cohomology imply that the cocycle of transition functions derived from any atlas is trivial. One can pick an atlas $\set{(U_{\alpha},\eta_{\alpha})}$ where the transitions are constant on each intersection. The triviality of the cocycle implies the existence of maps $g_{\alpha}:U_{\alpha}\to \mathrm{Ham}(W)$ so that $g_{\beta}g_{\alpha}^{-1}=g_{\alpha\beta}=\eta_{\beta}\eta_{\alpha}^{-1}$. Note that $g_{\alpha}$ are not required to be constant. Consider $g_{\alpha}$ as inducing maps $g_{\alpha}:U_{\alpha}\times W\to U_{\alpha}\times W$.

  In the following, suppose that the atlas uses a good open cover as in \cite[\S5]{bott_tu}, i.e., suppose that every finite intersection of open sets in the cover is empty or contractible.

  Consider the map $E\to \Sigma\times W$ which on $\pi^{-1}(U_{\alpha})$ equals $g_{\alpha}^{-1}\eta_{\alpha}$. This is well-defined, since on the intersection $U_{\alpha}\cap U_{\beta}$ we have $g_{\beta}^{-1}\eta_{\beta}=g_{\alpha}^{-1}\eta_{\alpha}$. Since the charts $\eta_{\alpha}$ take $\mathfrak{H}$ to the standard flat connection, the induced connection on $U_{\alpha}\times \Sigma$ appears in the form:
  \begin{equation*}
    (g_{\alpha}^{-1})_{*}(\text{standard flat connection}).
  \end{equation*}
  As in \S\ref{sec:coordinate-changes-hamilt-conn}, such a connection is Hamiltonian for $\Omega_{\alpha}=g_{\alpha}^{*}(\pr^{*}\omega)=\pr^{*}\omega-\d \mathfrak{a}_{\alpha}$. Following \S\ref{sec:norm-cond}, the Hamiltonian functions $\mathfrak{a}_{\alpha}(\bd_{s}),\mathfrak{a}_{\alpha}(\bd_{t})$ are chosen to be normalized in each fiber $\set{z}\times W$.

  The connection two-forms $\Omega_{\alpha},\Omega_{\beta}$ necessarily agree on their overlap because: $$(g_{\alpha}^{-1})^{*}(\Omega_{\alpha}-\Omega_{\beta})=\pr^{*}\omega-(g_{\beta}g_{\alpha}^{-1})^{*}\pr^{*}\omega=\pr^{*}\omega-g_{\alpha\beta}^{*}\pr^{*}\omega=0,$$ using that $g_{\alpha\beta}$ is constant.

  Writing $\Omega_{\alpha}=\pr^{*}\omega-\d\mathfrak{a}_{\alpha}$, we have that $\lambda_{\alpha\beta}=\mathfrak{a}_{\alpha}-\mathfrak{a}_{\beta}$ is closed on each intersection. Since $\lambda_{\alpha\beta}$ vanishes on vertical direction, it is exact, and can be written as $\lambda_{\alpha\beta}=\d f_{\alpha\beta}$ where $f_{\alpha\beta}$ is an $\R$-valued function pulled back from the base $U_{\alpha}\cap U_{\beta}$ (again, using that $W$ is connected).

  The normalization conditions imply that $f_{\alpha\beta}$ is constant, as follows. Compute
  \begin{equation*}
    \pd{f_{\alpha\beta}}{x_{i}}=\lambda_{\alpha\beta}(\bd_{i})=\mathfrak{a}_{\alpha}(\bd_{i})-\mathfrak{a}_{\beta}(\bd_{i}).
  \end{equation*}
  If $W$ is compact, integrate this over the fiber to conclude that $f_{\alpha\beta}$ is constant. If $W$ is non-compact, use that the right-hand side is one-homogenous (in the ends of the fibers) while the left hand side is constant on each fiber; so that the left hand side must be zero. In particular, $\d f_{\alpha\beta}$ vanishes identically, and $\mathfrak{a}_{\alpha}=\mathfrak{a}_{\beta}$ holds on the overlap. Hence there is a globally defined $\mathfrak{a}$ so that $\mathfrak{H}$ is the Hamiltonian connection for $\Omega=\pr^{*}\omega-\d\alpha$.
\end{proof}

\subsubsection{Families of flat connections parametrized by their monodromy}
\label{sec:famil-flat-conn}

Pick Hamiltonian systems $\varphi_{0,\tau}$ and $\varphi_{1,\tau}$, giving a family of homomorphisms $\mathrm{F}_{2}\to \mathrm{Ham}(W)$. As in the proof of Proposition \ref{prop:pop-construction}, this produces a bundle $E$ over $\Sigma\times [0,1]$, namely, the one obtained by a diagonal quotient of $[0,1]\times \Sigma'\times W$ by $\mathrm{F}_{2}$.

The restriction $E_{\tau}$ over $\Sigma\times \set{\tau}$ has an induced flat connection $\mathfrak{H}_{\tau}$ whose monodromies around $0$ and $1$ are given by $\varphi_{0,\tau}$ and $\varphi_{1,\tau}$.

The bundle over $\Sigma\times [0,1]$ can be trivialized (see Proposition \ref{prop:pop-construction}) and the argument in \S\ref{sec:vcech-cocycles} shows that we can express $\mathfrak{H}_{\tau}$ as the Hamiltonian connection associated to some family of one-forms $\mathfrak{a}_{\tau}$. The Cech-cohomology arguments can be done parametrically in $\tau$ so that one may assume that $\mathfrak{a}_{\tau}$ is smoothly varying.

One should note that a trivial bundle with flat Hamiltonian connection over a pair-of-pants actually determines a monodromy representation valued in the universal cover of $\mathrm{Ham}(W)$. Standard arguments in the theory of fiber bundles imply that the monodromy representation of $\mathfrak{H}_{1}$ is given by $x\mapsto (\varphi_{1,1},[\varphi_{1,t}])$ and $(\varphi_{1,1},[\varphi_{1,t}])$. In other words, the systems used to trivialize the bundle appear in the extension of the monodromy representation to the universal cover.

\subsubsection{Connectivity in space of flat connections on the pair of pants}
\label{sec:connectivity-space-flat}
The following proposition is used in \S\ref{sec:prod-oper-with}, and is of theoretical importance since it ensures that the operations on Floer cohomology defined using a pair-of-pants do not depend on the precise choice of Hamiltonian connection.

\begin{prop}
  Let $\Sigma$ be the pair-of-pants and let $\mathfrak{H}_{0},\mathfrak{H}_{1}$ be two Hamiltonian connections on $W\times \Sigma$ with the same monodromy representation in the universal cover of $\mathrm{Ham}(W)$. Then there is a smooth deformation $\mathfrak{H}_{\tau}$ of Hamiltonian connections interpolating from $\mathfrak{H}_{0}$ to $\mathfrak{H}_{1}$, with a fixed monodromy representation.
\end{prop}
\begin{proof}
  Since $\mathfrak{H}_{0},\mathfrak{H}_{1}$ define the same monodromy representation in the universal cover, there is a smooth map $g:\Sigma\to \mathrm{Ham}(W)$ so that $g_{*}\mathfrak{H}_{0}=\mathfrak{H}_{1}$ and which admits a lift to the universal cover. Since $\Sigma$ is homotopy equivalent to a wedge of circles and the universal cover is simply connected, there is a deformation $g_{\tau}$ so $g_{1}=g$ and $g_{0}=\id$. Then $(g_{\tau})_{*}\mathfrak{H}_{0}$ is the desired family of connections.
\end{proof}

\subsubsection{Cylindrical ends}
\label{sec:cylindrical-ends}

Let $\Sigma=[a,b]\times \R/\Z$ be a cylinder and suppose that $\mathfrak{H}$ is a flat Hamiltonian connection on $\Sigma\times W$. Introduce the notation $\mathfrak{H}(\varphi_{t})$ for the standard flat connection with $\mathfrak{a}=H_{t}\d t$ where $H_{t}$ is the normalized generator of $\varphi_{t}$.

The results of this section explain how to use the coordinate changes described in \S\ref{sec:coordinate-changes-hamilt-conn} to make $\mathfrak{H}$ appear in the standard form $\mathfrak{H}(\varphi_{t})$ in a given cylindrical end.

\begin{prop}\label{prop:coordinate-change-cyl}
  If the monodromy of $\mathfrak{H}$ around any circle $\set{s_{0}}\times \R/\Z$, in the universal cover, is represented by a conjugate of the system $\varphi_{t}$, then one can find a contractible family $g:\Sigma\to \mathrm{Ham}(W)$ so that $g^{-1}_{*}\mathfrak{H}=\mathfrak{H}(\varphi_{t})$.
\end{prop}
\begin{proof}
  Suppose $a=0$, $b=1$ and $s_{0}=0$ for simplicity. Let $\psi_{s,t}$ be the monodromy of $\mathfrak{H}$ along the path $\set{s}\times [0,t]$, and let $\kappa_{s,t}$ be the monodromy along the path $[0,s]\times \set{t}$.

  By assumption, there is a homotopy $\varphi^{\tau}_{t}$ starting at $\varphi_{0,t}=\rho\varphi_{t}\rho^{-1}$ so that $\varphi_{1,t}=\psi_{0,t};$ the homotopy has fixed endpoints at $t=0,1$.

  Let $g_{\tau}(s,t)=\kappa_{\tau s,t}\circ \varphi^{\tau}_{t}\circ \rho\circ \varphi_{t}^{-1}$. Then $g_{1}(s,t)=f_{s,t}\circ \varphi_{t}^{-1}$ where $f$ satisfies: $$\mathfrak{H}=f_{*}\mathfrak{H}(\id).$$
  On the other hand, since:
  \begin{equation*}
    (\varphi_{t}^{-1})^{*}\pr^{*}\omega=\pr^{*}\omega-\pr^{*}\omega(-,X_{t})\wedge \d t=\pr^{*}\omega-\d(H_{t}\d t),
  \end{equation*}
  it follows that $(\varphi)_{*}\mathfrak{H}(\id)=\mathfrak{H}(\varphi)$ and thus $g_{*}\mathfrak{H}(\varphi)=\mathfrak{H}$. Since $g_{0}=\rho$ is constant and $\mathrm{Ham}(W)$ is connected, $g=g_{1}$ is a contractible family, as desired.
\end{proof}

\subsubsection{Families of cylindrical ends}
\label{sec:famil-cylindr-ends}

The construction in \S\ref{sec:cylindrical-ends} can be done parametrically; i.e., if $\mathfrak{a}_{\tau}$ is a smooth family of connection one-forms for the cylinder and $\varphi_{\tau,t}$ is a smooth family of systems representing the monodromy around a circle $\set{s_{0}}\times \R/\Z$, one can find a smooth family $g_{\tau}$ so that $g_{\tau,*}^{-1}\mathfrak{H}=\mathfrak{H}(\varphi_{\tau})$.

Let $\Sigma$ be the pair-of-pants, and consider three cylindrical ends $\Sigma_{0},\Sigma_{1},\Sigma_{\infty}$ around the three punctures. The parametric construction in \S\ref{sec:famil-flat-conn} takes as input two paths $\varphi_{0,\tau},\varphi_{1,\tau}$ and produces a family of flat connections $\mathfrak{H}_{\tau}$ on $\Sigma$ (described by $\mathfrak{a}_{\tau}$). Applying a cut-off version of the coordinate changes from \S\ref{sec:cylindrical-ends} in each end allows us to assume that $\mathfrak{H}_{\tau}$ is given by $H_{0,t}\d t$, $H_{1,t}\d t$, respectively, in each end $\Sigma_{0},\Sigma_{1}$.

In the end $\Sigma_{\infty}$, the monodromy representation is conjugate to $\varphi_{\infty,t}=\varphi_{0,t}\circ \varphi_{1,t}$ (note that the reverse composition is in the same conjugacy class). Thus we can use the coordinate change from \S\ref{sec:cylindrical-ends} to assume that $\mathfrak{H}_{\tau}=\mathfrak{H}(\varphi_{\infty,t})$ in $\Sigma_{\infty}$.

A slight variation used in \S\ref{sec:prod-oper-with} is the following: let $\kappa:[0,1]\to [0,1]$ be a smooth cut-off function so $\kappa=0$ in a neighborhood of $0$ and $\kappa=1$ in a neighborhood of $1$. Then $\varphi_{\infty,t}$ and $\varphi_{\infty,\kappa(t)}$ have the same time-1 map in the universal cover, and so by the coordinate changes in \S\ref{sec:hamilt-conn-trivial} we can assume that $\mathfrak{H}_{\tau}$ appears as $\mathfrak{H}(\varphi_{\infty,\kappa(t)})$ in $\Sigma_{\infty}$.

\subsection{Floer's equation and Hamiltonian connections}
\label{sec:floers-equat-hamilt}

Let $\mathfrak{H}$ be a Hamiltonian connection on $W\times \Sigma\to \Sigma$ where $(\Sigma,j)$ is a Riemann surface. Pick a $\Sigma$-dependent $\omega$-tame complex structure $J$ on $TW$. Associated to these choices, let $J^{\mathfrak{H}}$ be the unique almost complex structure on $W\times \Sigma$ so that:
\begin{enumerate}
\item the fibers $W\times \set{z}\subset W\times \Sigma$ are almost complex submanifolds,
\item $\mathfrak{H}$ is a $J^{\mathfrak{H}}$-line, and,
\item the projection $(\mathfrak{H},J^{\mathfrak{H}})\to (T\Sigma,j)$ is complex-linear.
\end{enumerate}

A smooth map $u:\Sigma\to W$ is said to solve \emph{Floer's equation with data} $(\mathfrak{H},J)$ provided the induced section $z\mapsto (z,u(z))\in \Sigma\times W$ is $J^{\mathfrak{H}}$-holomorphic. The moduli space of all solutions is denoted $\mathscr{M}(\mathfrak{H},J)$.

See \cite[\S1.4.C',\S2.2]{gromov85} and \cite[\S8]{mcduffsalamon} for related discussion.

In the main body of the text, we only consider the case when $J$ is fixed; although we implicitly consider one-parameter variations of $J$ in \S\ref{sec:indep-choice-almost}.

\subsubsection{Energy density for Floer's equation}
\label{sec:energy-dens-floers}

Let $\Pi_{\mathfrak{H}}:TE\to TW$ be the projection whose kernel is $\mathfrak{H}$. Define the \emph{energy density two-form} for Floer's equation by the formula $\omega(\Pi_{\mathfrak{H}}\d u,\Pi_{\mathfrak{H}}\d u)$. It is straightforward to show that the energy density is everywhere non-negative (with respect to the complex orientation of $\Sigma$).

Define the \emph{energy} $E(u)$ to be the integral of the energy density two-form. Note that the energy of a solution $u$ is zero if and only if $u$ is a flat section.

The prototypical example is when $\Sigma=\R\times \R/\Z$, $\mathfrak{a}=H_{t}\d t$, so that the energy density is given by $\norm{\bd_{s}u}_{J}^{2}$ and flat sections are $u(s,t)=\gamma(t)$ where $\gamma(t)$ is an orbit of the system $H_{t}$.

\subsubsection{Energy identity for Floer's equation}
\label{sec:energy-ident-floers}

In local holomorphic coordinates $s+it$, so $\mathfrak{a}=K\d s+H\d t$, we have:
\begin{equation*}
  \Pi_{\mathfrak{H}}=\d\pr_{W}-X_{K}(u)\d s-X_{H}(u)\d t,
\end{equation*}
since $\bd_{s}+X_{K}$ and $\bd_{t}+X_{H}$ are tangent to $\mathfrak{H}$. In particular the local contribution to the energy is given by:
\begin{equation*}
  E(u)=\int \omega(\bd_{s}u-X_{K},\bd_{t}u-X_{H}).
\end{equation*}
Standard computations gives:
\begin{equation*}
  E(u)=\int u^{*}\omega-\bd_{s}(H(u))+(\bd_{s}H)(u)+\bd_{t}(K(u))-(\bd_{t}K)(u)+\omega(X_{K},X_{H})\d s\d t.
\end{equation*}
Simplifying, one obtains:
\begin{equation}\label{eq:energy-general-non-flat}
  E(u)=\int u^{*}\omega-u^{*}\d\mathfrak{a}+ u^{*}\mathfrak{r},
\end{equation}
where $\mathfrak{r}=(\bd_{s}H-\bd_{t}K+\omega(X_{K},X_{H})) \d s\d t$ is the curvature two-form valued in Hamiltonian functions (generating the curvature two-form valued in vector fields); see \eqref{eq:curvature_formula}. Patching together these local contributions proves the full energy is given by the same formula. See \cite[Lemma 8.1.6]{mcduffsalamon} and \cite[\S8g]{seidel_book} for similar identities.

As a consequence, if $\mathfrak{H}$ is a flat Hamiltonian connection and $\mathfrak{a}$ is normalized (so $\mathfrak{r}=0$) then we have the \emph{energy identity for a flat Hamiltonian connection}:
\begin{equation}\label{eq:flat}
  E(u)=\omega(u)-\int_{\bd\Sigma}u^{*}\mathfrak{a}.
\end{equation}

Here we treat $\bd\Sigma$ as a ``boundary in the sense of currents,'' i.e., as a formal object satisfying $\int_{\bd\Sigma}a=\int_{\Sigma}\d a$ for one-forms $a$. For example, if $\Sigma=\R\times \R/\Z$, and $\mathfrak{H}=\mathfrak{H}(\varphi_{t})$, then we recover the standard energy identity:
\begin{equation*}
  E(u)=\omega(u)+\int H_{t}(\gamma_{-})-H_{t}(\gamma_{+})\,\d t=\mathscr{A}_{\varphi_{t}}(\gamma_{-})-\mathscr{A}_{\varphi_{t}}(\gamma_{+}),
\end{equation*}
noting that the action difference is independent of the choice of capping of $\gamma_{+}$.

\subsubsection{Coordinate changes and Floer's equation}
\label{sec:coord-changes-energy}

Let $\Sigma$ be a surface with boundary $\bd\Sigma$, and let $\mathfrak{H}$ be a Hamiltonian connection over $\Sigma$. Let $L\subset W$ be a Lagrangian, and abbreviate $L=L\times \partial\Sigma$ so that we also think of $L$ as a subset of the total space.

If $g:\Sigma\to \mathrm{Ham}(W)$ is a homotopically trivial map, as in \S\ref{sec:coordinate-changes-hamilt-conn}, then $u$ solves Floer's equation for $\mathfrak{H},g_{*}^{-1}J$ with boundary values in $g^{-1}(L)$ if and only if $g(u)$ solves Floer's equation for $g_{*}\mathfrak{H},J$ with boundary values in $L$. This implies that energy bounds for solutions to $\mathscr{M}(\mathfrak{H},g_{*}^{-1}J,g^{-1}(L))$ are equivalent to energy bounds for $\mathscr{M}(g_{*}\mathfrak{H},J,L)$. This observation is used in \S\ref{sec:prod-oper-with}.

\bibliographystyle{alpha}
\bibliography{citations}
\end{document}